\documentclass{article}

\usepackage{amsmath, amsthm, amsfonts}
\usepackage{enumerate}
\usepackage{hyperref}
\usepackage{amssymb}
\usepackage{graphicx}
\usepackage[all]{xy}
\newtheorem{thm}{Theorem}[section]
\newtheorem{cor}[thm]{Corollary}
\newtheorem{lem}[thm]{Lemma}
\newtheorem{prop}[thm]{Proposition}

\theoremstyle{definition}
\newtheorem{defi}[thm]{Definition}
\theoremstyle{remark}
\newtheorem{nota}[thm]{Notation}
\newtheorem{rem}[thm]{Remark}
\usepackage{authblk}

\title{The  lower plus and upper plus constructions for fibered biset functors }
\author[]{J. Miguel Calderón \\
calderonl@matmor.unam.mx
}
\affil[1]{PCCM-UNAM-UMSNH\\
Morelia Michoacán Mexico }

\begin{document}
\maketitle
MSC: 19A22  20C15
20C20

https://orcid.org/0009-0007-4685-8837

\abstract

let $A$ be an abelian group, not necessarily finite. The main objective of  this paper  is to provide two constructions for a fibered $A$-biset functor. The first  is the lower plus construction, and the other is the upper plus  construction. These constructions coincide with the lower plus  and upper plus  constructions for biset functors (see \cite{fibered}) when the fiber is the trivial group  $\lbrace \cdot \rbrace$. To generate these constructions, it is necessary to have a family $\mathcal{G}$  of finite groups and a function $\mathcal{S}$  such that, for all $G, H \in \mathcal{G}$, we relate them to a set $\mathcal{S}(G, H)$ of
 $$\mathcal{M}^A (G,  H)= \mathcal{M}(G,H):=\lbrace (L,\varphi) \mid L\leq G\times H , \ \varphi \in Hom(L,A)\rbrace $$
The pair $\mathcal{(G,S)}$ 
 satisfies  certain axioms outlined in Definition  \ref{axioms}. These axioms are essential for enabling the pair $\mathcal{(G, S)}$ to generate the subcategory $\mathcal{D}$ of the category of fibered $A$-bisets, denoted as $\mathcal{C}^A.$. In the Section \ref{sect Fibered biset}, we recall and extend the basic definitions and results for fibered bisets and fibered biset functors. In Section  \ref{sect low and upper plus}, the upper and lower constructions are defined. For this purpose, the pair$\mathcal{(G, S)}$ is associated with $\mathcal{(G, S_+ )}$ and $\mathcal{(G, S^+ )}$, with which we can construct the subcategories $\mathcal{D}_+$ and  $\mathcal{D}^+$  of $\mathcal{C}^A$. In Sections \ref{sect functor lower } and \ref{sect functor upper }, the constructions $F_+$ and $F^+$ are described for a fibered $A$-biset functor  $F$ over $\mathcal{D}$, which are fibered biset functors over $\mathcal{D_+}$  and   $\mathcal{ D^+}$  respectively. In Section \ref{sect  mark morphism}, the mark morphism $F_+ \longrightarrow F^+$  is defined, which is a natural transformation. Necessary conditions on the ring $R$ are also provided for the mark morphism to be injective or even bijective. In Section \ref{sect Green funct} the additional condition that $F$ has a multiplicative structure is added, more precisely, that $F$ is a fibered Green functor over $\mathcal{D}$. Then it can be shown that $F_+$ and $F^+$ inherit the structure of fibered Green functors over $\mathcal{D_+}$ and  $\mathcal{D^+ }$, respectively, and the mark morphism is multiplicative. In Section \ref{example fibre}, the ring of characters is defined as a fibered biset functor, and its construction  $-_+$, which is the global functor of fibered representations, is calculated. Finally, in Section \ref{adjuncion}, it is provedd that the functor related to the lower construction, denoted by $-_+$, is the left adjoint for a restriction functor.

{\small{\bf Keywords}{:} 
Group, The character of a group, biset, biset functor, Fibered biset functor.
}

\section{Introduction}

In \cite{serge-biset} and \cite{bisetscategory}, Serge Bouc introduced and developed the theory of biset functors, which provides a framework for situations where structural such as
restriction, induction, inflation, and deflation, or a subset of them, are present.  Biset functors find applications in various representation rings, such as the Burnside ring, the character ring, the  Green ring, and the  trivial source ring of a finite group $G$.\\
 Furthermore, Rober Boltje, Alberto G Raggi-Cárdenas and  Luis Valero-Elizondo introduced and developed  the constructions lower plus and upper plus by  bisets functor in  \cite{constructios+}. These  constructions are the first step to extend the general framework of
canonical induction formulas, introduced by  Boltje (see  \cite{canonicalinductionformulae}).  the $-^+$-construction  and the
$-_+$-construction to the framework of biset functors will allow rewriting the theory of
canonical induction formulas and derive additional properties, as for example that they
commute not only with restrictions,  but also with inflation.\\
 On the other hand,  Dress introduced the theory of fibered bisets in \cite{dressRing}, and Rober Boltje and Olcay  Co{\c{s}}kun  further developed the theory of fibered biset functors in \cite{fibered}. This  is similar to theory of  biset functors, the motivation for this theory comes from the fact that representation rings of finite groups
carry more structure, by considering multiplication with one-dimensional representations as structural maps.\\
Let $A$ be a multiplicative written abelian group. The $A$-fibered Burnside ring $B^A(G)$
is defined as  the Grothendieck group of $A$-fibered $G$-sets, i.e., $G \times A$-sets with finitely many $A$-orbits
which are free as $A$-sets,  A  $\mathbb{Z}$-basis of $B^A(G)$ is parameterized  by $G$-conjugacy classes of pairs 
 $(H, \varphi)$, where $H$ is a subgroup of $G$ and $\varphi \in Hom(H, A).$ Similarly,  one defines
$A$-fibered $(G, H)$-bisets and their Grothendieck groups $B^A(G, H) := B^A(G \times H)$. The category fibered bisets allow a tensor product construction.  giving  rise to the $A$-fibered biset category  over a commutative ring $R$, denoted by $\mathcal{C}^A_R$, one remark is that  the $A$ -fibered l (also called monomial) Burnside ring is the $-_+$-construction applied to the biset functor mapping a finite group $G$ to the free $\mathbb{Z} $-module with basis $Hom(G, A)$.\\
There exists a natural embedding $B(G) \longrightarrow B^A(G)$ and  a  natural splitting map
$B^A(G) \longrightarrow B(G) $ of this embedding in the category of rings. This embedding allows us  to view the biset category as a subcategory of the $A$-fibered biset category. To view an $A$-fibered biset functor $F$ via restriction as a biset functor means forgetting some of its structure.

The main objective  of this paper is to introduce  two constructions, lower plus and upper plus, for an $A$-fibered biset  $F$. These constructions are specifically designed such that when $F$ is viewed as a biset functor via restriction, they coincide with the lower plus and upper plus constructions for biset functors.  We begin  by selecting a family $\mathcal{G}$ of finite groups. For every pair of groups $G$ and $H$ from $\mathcal{G}$,  a
choice of a set $\mathcal{S}(G, H)$ of $$\mathcal{M}^A (G,  H)= \mathcal{M}(G,H):=\lbrace (L,\varphi) \mid L\leq G\times H , \ \varphi \in Hom(L,A)\rbrace, $$ These choice of sets satisfy certain axioms that lead to the construction of a category $\mathcal{D}$, which is a subcategory of the $A$-fibered biset category $\mathcal{C}^A$.  \\
In the  Section \ref{sect Fibered biset} we provide a review of the  fundamental definitions and facts related to   fibered bisets and fibered biset functors. This section severs as a reminder of the basic concepts and results necessary for  the subsequent constructions. In the Section \ref{sect low and upper plus} traduces the construction of associated pairs  $ \mathcal{(G, S_+ )}$ and
$\mathcal{(G, S^+ )}$ from the given  pair $\mathcal{(G, S)}$, which lead to subcategories $\mathcal{D}_+$ and $\mathcal{D}^+$ of $\mathcal{C}^A$. In the Sections \ref{sect functor lower } and \ref{sect functor upper } describe the construction of the biset functor $F_+$ on $\mathcal{D_+}$ and the biset functor $F^+$ on $\mathcal{ D^+}$ associated to an $A$-fibered biset functor $F$ on $\mathcal{D}$. 
In the Section \ref{sect  mark morphism} we introduce  the mark morphism as a natural transformation from  $F_+$  to $F^+$. We prove that, under certain conditions on the base ring, the mark morphism is injective, or even bijective. In the Section \ref{sect Green funct} deals with the situation that $F$  is a Green  fibered biset functor over $\mathcal{D}$. We proved that both  $F_+$   and $F^+$  inherit Green fibered biset functor structure  over $\mathcal{D_+}$ and $\mathcal{D^+ }$ respectively. Furthermore, we establish that the mark morphism preserves the multiplicative structure. In the section \ref{example fibre} we introduce the global representation fibered ring functor, which is a fibered biset functor. We proved that this functor corresponds to the lower plus construction applied to the character ring viewed as a fibered biset functor. For end, in the Section \ref{adjuncion}, we prove that the functor $-_+$ is the  left adjunction  of one functor Restriction.

\section{Fibered bisets }\label{sect Fibered biset}
We start this section with notation that  will  be used  throughout this paper:  We fix a multiplicatively written abelian group $A$, where $R$ denotes a commutative ring with unity,  and  $G$, $H$, and $K$ denote finite groups.
We will now briefly review the concepts and fundamental results related to  $A$-fibered bisets, the $A$-fibered biset category,
and $A$-fibered biset functors, as presented in chapters 2 and 3 of \cite{serge-biset}] and Sections 1–3 of  \cite{fibered}.\\
For any finite group $G$, we define
 \begin{align*}
 G^\ast := Hom (G, A)
 \end{align*}
and consider  $G^\ast $ as an abelian group with point-wise multiplication. 
A left $A$-fibered $G$-set $X$ is a left $(G \times A)$-set with only finitely many $A$-orbits which are free as $A$-sets, i.e., $ax = bx$ for $a, b \in A$ and $x\in X$ implies
$a = b$.  Together with
$(G \times A)$-equivariant maps, one obtains a category $_Gset^A$. Similarly, we define the category $set^A_G$ of right $A$-fibered $G$-sets. We sometimes view a left $A$-fibered $G$-set $X$ also as a right $A$-fibered $G$-set. In this case, the right $A$-action is given by $xa := ax$ and the right $G$-action by $xg := g^{-1}x$ for $a\in A$, $g\in  G$ and $x \in  X$. We view the $A$-action in analogy to the action of the base ring $R$ when considering a module for an $R$-algebra. \\
In the category $_G\text{set}^A$, there is a natural notion of disjoint union, which serves as a categorical coproduct. Moreover, there is a tensor product $X \otimes_A Y \in {_Gset^A}$
defined as the set of $A$-orbits of the direct product $X \times Y$ under the $A$-action $a(x, y) = (xa^{-1}, ay)$. The orbit of $(x, y)$ is denoted by $x\otimes y$.  The gruops  $G$ and $A$   on the tensor product  $X\otimes_AY$ via the following actions:  $a(x\otimes y) = ax \otimes y = x\otimes ay$ and $g(x\otimes y) = gx\otimes gy$
for all $a \in A$, $g \in G$, $x \in X$ and $y \in Y$.\\
The Grothendieck group of $_Gset^A$ with the disjoint
union is denoted by $B^A(G)$, this group possesses a ring structure with $\otimes$ as its multiplication operation.
In $B^A(G)$, the class of $X \in {_Gset^A}$  is denoted by $[X]$.\\ 
 The category of $A$-fibered $(G, H)$-bisets is formally defined as the category   $_{G\times H}set^A$, which  is also  denoted by $_Gset^A_H$.  An $A$-fibered $(G, H)$-biset $X$ is often considered to be equipped with a left
$G$-action, a right $H$-action and a two-sided $A$-action, with  all four actions commuting with each other, his is denoted using fibered biset notation as $gaxhb := ((g, h^{-1}
), ab)x$ for $g \in G$, $h \in H$, $a, b \in A$, and $x \in X$ (fibered biset notation).\\
The tensor product of $A$-fibered bisets is defined as follows: Let $X$ be an object in $_G\text{set}^A_H$ and $Y$ be an object in $H\text{set}^A_K$. The tensor product $X \otimes_{AH} Y$ consists of those $(H\times A)$-orbits of $X \times Y$ under the action $(h,a)(x,y) = (x(ha)^{-1},hay)$ that are $A$-free under the induced $A$-action on these orbits. The orbit of $(x,y)$ is denoted by $x\otimes y$. In fibered biset notation, $xah \otimes y = x\otimes ahy$. The set $X \otimes_{AH} Y$ is an $A$-fibered $(G, K)$-biset via $(g, k, a)(x\otimes y) = (g, a)x\otimes ky$ in formal notation and $ga(x\otimes y)k = gax\otimes yk$in fibered biset notation. 
This construction of the tensor product is associative, meaning that the map  $(x\otimes  y)\otimes z \longrightarrow x\otimes (y \otimes z)$ is
a well-defined isomorphism in $_Gset^A_L$ between $(X \otimes_{AH} Y )\otimes _{AK} Z$ and $X \otimes_{AH} (Y \otimes_{AK} Z)$, whenever
$L$ is a finite group and $Z \in {_Kset^A_L}$. It is functorial in $X$, $Y$ and $Z$.\\

The Grothendieck group of $_Gset^A_H$ with respect to disjoint union,  is denoted  by $B^A(G, H)$ Additionally, we define  $B^A_R(G, H) := R\otimes_\mathbb{Z} B^A(G, H)$.
where $R$ is a commutative ring with unity. The class of an $A$-fibered $(G, H)$-biset $X$ in $B^A(G, H)$ is denoted by $[X]$. We also write $[X]$ for $1 \otimes [X] \in B^A_R(G, H)$.

For a finite group $G$, we denote by $\mathcal{M}^A(G)$ the set of all pairs of the form  $(K, \kappa)$, where $K$ is a subgroup of $G$ and $\kappa \in K^\ast$. In many cases, we will simply write $\mathcal{M}(G)$ instead of $\mathcal{M}^A(G)$. The set $\mathcal{M}(G)$ has a poset structure defined by $(L,\phi) \leq (K, \kappa)$ if $L \leq K$ and $\phi = \kappa\vert_L$. Moreover, $G$ acts on $\mathcal{M}(G)$ by conjugation, denoted by
\begin{align*}
^g  (K, \kappa):= (^g K, ^g \kappa)
\end{align*}
for $g\in G$,  where $^g\kappa : ^g K \longrightarrow A$ such that  $^g\kappa (^g k) = \kappa (k)$ for all  $k\in K$.
Let $H$ be another finite group, and consider $(U, \varphi) \in \mathcal{M}(G \times H)$. We denote the projection maps $p_1 : G \times H \longrightarrow G$ and $p_2 : G \times H \longrightarrow H$. Furthermore, we define
\begin{align*}
k_1(U) = \lbrace g\in G \mid (g,1) \in U \rbrace \ \ \text{ and } \ \ k_2(U) = \lbrace h\in H \mid (1,h)\in U \rbrace 
\end{align*}
We define 
\begin{align*}
\varphi_1 : k_1(U) \longrightarrow A&  &\text{ and }&  &\varphi_2 : &k_2(U) \longrightarrow A\\
\text{ } g \longmapsto \varphi((g,1))&   &\text{ }& &\text{   } \  \ &h\longmapsto \varphi((1,h))^{-1}.
\end{align*}
Let  $V$  a subgroup of $H\times K$, one sets
\begin{align*}
U \ast V := \lbrace(g, k) \in  G \times K \mid \exists h \in H \ : \  (g, h) \in  U, (h, k) \in V \rbrace
\end{align*}
a subgroup of $G \times K$. Moreover, if $(V,\psi) \in \mathcal{M}(H\times K) $
 with $$\varphi_2|_{k_2(U)\cup\cap k_1(V )} = \psi_1|_{k_2(U)\cup\cap k_1(V )},
$$ then
one obtains a homomorphism $\varphi \ast \psi \in Hom(U \ast V, A)$ defined by
\begin{align*}
(\varphi \ast \psi)(g, k) := \varphi(g, h)\psi(h, k),
\end{align*}
where $h \in H$ is chosen such that $(g, h) \in U$ and $(h, k)\in V$.\\
\textbf{Stabilizing pairs}. Let $X$ be an $A$-fibered $(G, H)$-biset. We will denote the $A$-
orbit of the element $x\in  X$  by $[x]$. Note that $G \times H$ acts on the set of $A$-orbits of $X$ by
$(g, h)[x] := [(g, h)x]$. For $x \in X$, let $(G\times H)_x \leq G \times H$ denote the stabilizer of $[x]$.\\
We define a map
\begin{align*}
X\longrightarrow &\mathcal{M}(G\times H)\\
x\longmapsto &((G\times H)_x, \phi_x),
\end{align*}
with
\begin{align*}
\phi_x : (G\times H)_x \longrightarrow A
\end{align*}
determined by the equation
\begin{align*}
(g, h)x = \phi_x (g, h)x
\end{align*}
for $(g, h)\in (G\times H)_x$ . We call $((G\times H)_x , \phi_x )$ the stabilizing pair of $x$.\\
Let $X$ be an $A$-fibered $(G, H)$-biset. It is
clear that the $A \times G \times H$-action on $X$ is transitive if and only if the $G \times H$-action
on the set of $A$-orbits of $X$ is transitive. In this case, we call $X$ a transitive $A$-fibered
$(G, H)$-biset. Moreover, if $X$ is a  transitive $A$-fibered
$(G, H)$-biset,   we have the next isomorphism of $A$-fibers  $(G, H)$-biset:

\begin{align*}
X\cong \frac{G\times H \times A}{ \lbrace (g,h,\phi_x((g,h))^{-1}) \mid (g,h)\in (G\times H)_x \rbrace}:=\dfrac{G\times H}{(G\times H)_x, \phi_x}.
\end{align*}
for all $x\in X$.

\begin{thm}[Corollary 2.5 of \cite{fibered}]\label{mackey}
For $(U, \varphi)\in \mathcal{M}(G \times H)$ and $(V, \psi) \in \mathcal{M}(H \times K)$ one
has
\begin{align*}
\left[  \frac{G\times H}{ U, \varphi}\right] \otimes_{AH}\left[ \frac{H\times K}{V, \psi }\right] =\sum_{\substack{ t\in [p_2(U)\backslash H/p_1(V)]\\ \varphi_2 \vert_ {H_t}={^t\psi_2\vert_{H_t} }}} \left[ \frac{G\times K}{U\ast {^{{(t,1)} V}}, \varphi \ast {^{{(t,1)} \psi}}}\right] 
\end{align*}
where $H_t = k_2(U) \cap ^tk_1(V)$.
\end{thm}

We define the $A$-fibered biset category $\mathcal{C}^A_R$ as the category whose objects are finite groups, and whose morphisms are given by $$Hom_{\mathcal{C}^A_R}(H, G) := B^A_R(G, H).$$ The composition of morphisms is defined as the tensor product of $A$-fibered bisets. The identity morphism of an object $G$ is represented by the element $\left[G\times G/ \Delta(G), 1\right] \in B^A_R(G,G)$, here, $\Delta(G) := \lbrace (g, g) | g \in G \rbrace$. Which also corresponds to the class of the $A$-fibered $(G,G)$-biset $G\times A$ with the multiplication actions from both sides.

In many cases, we will use $\mathcal{C}^A$ as a shorthand for $\mathcal{C}^A_R$.

Next, we define an elementary fibered biset as follows: For $K \leq G$, we set:
\begin{align*}
Res^G_K:= \left[  \frac{K\times G}{\Delta(K), 1} \right]&  &\text{ and }& &I{nd^G_K}:=\left[  \frac{G\times K}{\Delta(K), 1} \right].
\end{align*}
Let $N \unlhd G$ and  $\pi: G \longrightarrow G/N$ the canonical projection, we set:
\begin{align*}
Def^G_{G/N} := \left[ \frac{G/N \times G}{^{(\pi,1)} \Delta(G)} \right]&  &\text{ and }&  &Inf^G_{G/N} :=\left[  \frac{G \times G/N}{^{(1,\pi)} \Delta(G)} \right] 
\end{align*}
where $^{(\pi,1)} \Delta(G):= \lbrace (gN, g) \mid g\in G \rbrace$ and $^{(1, \pi)} \Delta(G):= \lbrace (g, gN) \mid g\in G \rbrace$.
 Finally, if $f : H \longrightarrow G$ is an isomorphism, one sets
 \begin{align*}
 iso(f):= \left[ \frac{G\times H}{ \lbrace(f(h), h) \mid h\in H \rbrace, 1} \right]. 
 \end{align*}
\begin{thm}[Proposition 2.8 of \cite{fibered}]
Let $(U, \varphi) \in \mathcal{M}(G\times H)$ and set $P := p_1(U)$, $Q := p_2(U)$,
$K := ker(\varphi_1)$, and $L := ker(\varphi_2)$. Then $K \unlhd P$, $L \unlhd Q$, $K \times L  \unlhd U$, and
\footnotesize{
\begin{align*}
\left[ \frac{G\times H}{ U, \varphi}\right] = Ind^G_P \otimes_{AP} Inf^P_{P/K} \otimes_{A(P/K)} \left[\frac{P/K\times Q/L}{U/(K\times L), \overline{\varphi}} \right] \otimes_{A(Q/L)} Def^Q_{Q/L} \otimes_{AQ} Res^H_Q 
\end{align*}}\normalsize{
where $\overline{\varphi} \in  (U/(K \times L))^\ast$
is induced by $\varphi$ and $U/(K \times L)$ is viewed as a subgroup of $P/K \times Q/L$ via
the canonical isomorphism $(P \times Q)/(K \times L)\cong P/K \times Q/L$.}
\end{thm}
\begin{lem}\label{estProd}
Let  $G$, $H$ and  $K$ be finite groups. Let  $U$ be an  $A$-fibered $(G,H)$-biset  and let  $V$ be an $A$-fibered    $(H,K)$-biset, then 
\begin{align*} 
((G\times K)_{u\otimes v}, \phi_{u\otimes v})=((G\times H)_u,\phi_u) \ast ((H\times K)_v, \phi_v)
\end{align*}
for all  $u\otimes v\in U \otimes_{AH} V$.
\end{lem}

\section{The lower plus  and upper plus  constructions on subcategories $\mathcal{D}$ of $\mathcal{C}$} \label{sect low and upper plus}
The objective of this section is to build the subcategories  $\mathcal{D_+}$ and $\mathcal{D}^+$ of the category of fibered bisets.These categories are derived from a subcategory $ \mathcal { D } $ of the category of fibered bisets, which is defined as the smallest subcategory containing $ \mathcal { D } $ and the morphisms $ Res $ or $ Ind $. This construction is a generalization of the constructions in [ \cite{ constructios+}. Section 3]\\
If $\mathcal{ G}$ is a class of finite groups. For  $G\in \mathcal{G}$,  we define 
\begin{align*}
\Sigma_{\mathcal{G}} (G) :=\lbrace H\leq G\mid H\in \mathcal{G} \rbrace.
\end{align*}

\begin{defi}\label{axioms}
The data $(\mathcal{G, S})$, consists of $ \mathcal { G } $, which is a class of finite groups, and $ \mathcal { S } $, a family denoted by $ \mathcal { S } =( \mathcal { S } (G,H))_ { G,H \in \mathcal { G } } $. Here, $ \mathcal { S } (G,H) \subseteq \mathcal { M } (G \times H) $ for $ G, H \in \mathcal { G } $. Moreover, we assume that $ ( \mathcal { G, S } ) $ satisfies the following axioms when necessary:
\begin{enumerate}[(i)]
\item For all $G\in \mathcal{G}$,  one has  $(\Delta(G), 1) \in \mathcal{S} (G,G)$.
\item For all  $G$, $H \in \mathcal{G}$ the set $\mathcal{S}(G,H)$ is closet under $(G\times H)$-conjugations.
\item For all  $G, H, K \in \mathcal{G}$ and all  $(V,\phi)\in \mathcal{S}(G,H)$ and $(U,\psi)\in \mathcal{S}(H,K)$ such that 
\begin{align*}
\phi_2\vert_{K_{2}(V)\cap k_1{(U)}}=\psi_1\vert_{K_{2}(V)\cap k_1{(U)}},
\end{align*}
one has  $(V\ast U, \phi \ast \psi) \in \mathcal{S}(G,K)$.
\item For all  $G, H \in \mathcal{G}$, all  $(D,\phi)\in \mathcal{S}(G,H)$ and all  $K\in \sum_{\mathcal{G}}(H)$, one has  $D\ast K \in \mathcal{G}$ and   $(D, \phi) \ast ( \Delta(K), 1)\in \mathcal{S}(D\ast K, K)$.
\item For all  $G\in \mathcal{G}$ and  all $H\in \sum_{\mathcal{G}} (G)$, one has $(\Delta(H),1)\in \mathcal{S}(G,H)$.
\item  For all  $G\in \mathcal{G}$ and all $H\in \sum_{\mathcal{G}} (G)$, one has $(\Delta(H),1)\in \mathcal{S}(H,G)$.
\item  For all $G, H \in \mathcal{G}$ and all  $(D,\phi)\in \mathcal{S}(G,H)$, one has $p_2(D)\in \mathcal{G}$, and for all   $K\in \sum_{\mathcal{G}}(p_2(D))$ one has    $D\ast K \in \mathcal{G}$ and   $(D, \phi)\ast (\Delta(K),  1)$ is element of $ \mathcal{S}(D\ast K, K)$.
\item  For all  $G, H \in \mathcal{G}$ and all  $(D,\phi)\in \mathcal{S}(G,H)$, one has $p_1(D)\in \mathcal{G}$, and for all   $K\in \sum_{\mathcal{G}}(p_1(D))$ one has    $K\ast D \in \mathcal{G}$ and   $(\Delta(K) ,1)\ast (D,  \phi )$ is element of $ \mathcal{S}(K, K\ast D)$.
\end{enumerate} 
In addition, we will say that the pair $\mathcal{(G,S)}$ satisfies condition $k_2$ if, for any groups $G$ and $H$ from $\mathcal{G}$ and for every $(D, \phi) \in \mathcal{S} (G,H)$, it holds that $k_2(D)=\lbrace 1\rbrace$.
\end{defi}
For any  $G \in \mathcal{G}$.  we set 
\begin{align*}
\mathcal{ M}^\prime (G):= \lbrace (H, \lambda) \mid H \in \sum_\mathcal{G} G, \lambda \in H ^\ast \rbrace
\end{align*}
the set $\mathcal{M}^\prime (G)$  is also a poset  with the same  structure of $\mathcal{M} (G).$

\begin{rem} \label{conj}
Let  $(\mathcal{G}, \mathcal{S})$  be such that they satisfy axioms $(i)$ to $(iii)$ and additionally Axiom $(iv)$  or Axiom $(vii)$.  If $G\in \mathcal{G}$, $H\in \sum_{\mathcal{G}}(G)$ and  $g\in G$.  Then  $^g H \in \mathcal{G}$ and $^{(g,1)}(\Delta (H), 1) \in  \mathcal{S}(^gH.H)$.   
\begin{proof}
By Axioms  $(i)$ and  $(ii)$ one has  $^{(g,1)}(\Delta(G),1) \in \mathcal{S}(G,G)$.
Now by Axiom $(iv)$ or $(vii)$ one has 
\begin{align*}
^{(g,1)}\Delta(G)\ast H= {^gH }\in \mathcal{G}
\end{align*}
we obtain 
\begin{align*}
(^{(g,1)}\Delta(G)\ast \Delta(H),^{(g,1)} 1\star 1)= (^{(g,1)}\Delta(H),^{(g,1)} 1) \in \mathcal{S}(^gH,G).
\end{align*}
\end{proof}
\end{rem}
\begin{defi} \label{S}
Let  $(\mathcal{G},\mathcal{S})$  be a set of data satisfying  Axioms $(i)$ to $(iii)$. We define the subcategory   $\mathcal{D=C}^A(\mathcal{G}, \mathcal{S})$ of the $A$-fibered biset category $\mathcal{C}^A$ as follows: The objects of  $\mathcal{D}$ are  $\mathcal{G}$. For any   $G, H \in \mathcal{G}$, 
\begin{align*}
Hom_\mathcal{D}(G,H):=R\otimes_{\mathbb{Z}}\bigg\langle \bigg\lfloor \dfrac{G\times H}{(U,\phi)} \bigg\rfloor \mid (U,\phi)\in \mathcal{S}(G,H) \bigg\rangle_{\mathbb{Z}} \subseteq B^A_R(G,H)
\end{align*}
Let's note that by axiom $(i)$, the element $\left[ \frac{G\times G}{ (\Delta(G), 1)} \right]$ belongs to $Hom_\mathcal{D}(G,G)$, which is the identity morphism of $G$. Furthermore, by axioms $(ii)$ and $(iii)$ and Theorem \ref{mackey}, we have that $\mathcal{D}$ is closed under the composition of $\mathcal{C}^A$.

\end{defi}
Note that, if  $(\mathcal{G},\mathcal{S})$ additionally satisfies the Axiom $(v)$ (resp. $(vi)$) of \ref{axioms}, it  ensures that the category $\mathcal{D}$ contains all possible inductions (resp. restrictions). We say that $(\mathcal{G},\mathcal{S})$ satisfies  the condition $k_2$ if, for any $G, H \in \mathcal{G}$ and $(U,\phi) \in \mathcal{S}(G,H)$,  one has $k_2(U)=\lbrace e \rbrace$.
\begin{defi}\label{defi S_+}
Let  $G, H \in \mathcal{G}$, we define $\mathcal{S_+}(G,H)$ as the set of $\mathcal{M}(G\times H)$ such that 
\begin{align*}
\mathcal{S}_{+}(G,H):=\lbrace (D,\phi) \in \mathcal{M}(G\times H) \mid  P_1(D)\in \mathcal{G}, \ \ (D,\phi)\in \mathcal{S}(P_1(D), H)  \rbrace.
\end{align*}

\end{defi}
Now, we will prove that if  $(\mathcal{G},\mathcal{S})$  also satisfies Axiom $(iv) $,  then  $(\mathcal{G},\mathcal{S}_{+})$  satisfying Axioms $(i)$, $(ii)$ and $(iii)$,  so that we obtain a category  \-$  \mathcal{D_+ }:=\mathcal {C}^A \mathcal{(G, S_+)}$.
Let  $G,H, K \in \mathcal{G}$:

\begin{enumerate}[(i)]
\item   Clearly,  $(\mathcal{G, S_{+}})$ satisfies Axiom $(i)$.
\item  Let $(D,\phi)\in \mathcal{S}_{+}(G,H) $ and  $(g,h)\in G\times H$. We have $p_1(D)\in \mathcal{G} $  and  $(D,\phi)\in \mathcal{S}(p_1(D),H)$,  by Axiom $(iii)$ of  $ \mathcal{(G,S)}$,  $$^{(1,h)} (D,\phi) \in \mathcal{S}(p_1(D),H).$$
By remark \ref{conj} one has:
\begin{itemize}
\item $^gp_1(D)=p_1(^{(g,h)} D)\in \mathcal{G}$ .
\item $^{(g,1)}(\Delta(p_1(D)), 1)\in \mathcal{S}(^gp_1(D), p_1(D))$.
\end{itemize}
Thus 
\begin{align*}
^{(g,1)}(\Delta(p_1(D)), 1)\ast {^{(1,h)} (D,\phi)}\in \mathcal{S} (^g p_1(D),H)
\end{align*}
note that
\begin{align*}
^{(g,1)}\Delta(p_1(D) \ast {^{(1,h)}D}={^{(g,h)} D}.
\end{align*}
Let  $(l, s)\in  D $, then
\begin{align*}
({^{(g,1)} 1 \ast {^{(1,h)} \phi }})(^g l, ^h s)= {^{(g,1)} 1}((^gl,l))^{(1,h)} \phi ((l,^hs))= \phi(l,s).
\end{align*}
Thus  $ ^{(g,h)} (D,\phi)\in \mathcal{S} (^g p_1(D),H)$ this means that
$ ^{(g,h)} (D,\phi)\in \mathcal{S}_{+} (G,H)$.
\item Let $(D, \phi)\in \mathcal{S}_+(G,H)$ and  $(T, \psi)\in \mathcal{S}_+(H,K)$ 
such that $\phi_2=\psi_1$ in  $k_2(D)\cap k_1(T)$,  by definition of,  $\mathcal{S}_+$ one has:
\begin{itemize}
\item $p_1(T)$ and  $p_1(D)$  are elements of $\mathcal{G}$.
\item $(D,\phi)\in \mathcal{S}(p_1(D), H) $ and  $(T,\psi)\in \mathcal{S}(p_1(T), K) $.
\end{itemize}
Since  $(\mathcal{G}, \mathcal{S})$ satisfies Axiom $(iv)$, one has:
\begin{itemize}
\item  $p_1(D\ast T)=D\ast p_1(T)\in \mathcal{G}$
\item  $(D,\phi)\ast (\Delta(p_1(T)), 1) \in \mathcal{S}(p_1(D\ast T), p_1(T) )$
\end{itemize}
thus
\begin{align*}
\left( (D,\phi)\ast (\Delta(p_1(T)), 1)\right) \ast (T,\psi)=(D\ast T, \phi \ast \psi)\in \mathcal{S}(p_1(D\ast T), K).
\end{align*}
\end{enumerate}
Then  $(\mathcal{G}, \mathcal{S}_+)$ satisfies   Axioms $(i)$, $(ii)$ and $(iii)$.

\begin{rem}
Note that each $(D,\phi)\in \mathcal{S}_+(G, H)$,  can be written as  $(\Delta(p_1 (D)),1) \ast (D,\phi)$, 
with  $(\Delta(p_1(D),1)\in \mathcal{S_+}(G,  p_1(D))$ and $ (D,\phi)\in \mathcal{S}_+(p_1(D), H)$ Thus  

\begin{align*}
\left[ \frac{G\times H}{(D, \phi)}\right]=\left[ \frac{G\times p_1(D)}{(\Delta (p_1(D)), 1)} \right] \otimes_{Ap_1(D)} \left[ \frac{p_1(D)\times H}{(D,\phi)} \right] 
\end{align*}
in the category $\mathcal{D}_+$.
\end{rem}

\begin{prop} \label{D_+}

The data $(\mathcal{G},\mathcal{S})$ satisfying  Axioms $(i)$ to  $(iv)$.
\begin{enumerate}[(a)]
\item For  $G, H \in \mathcal{G}$ and   $(D,\phi)\in \mathcal{M}(G,H)$, we have   $(D,\phi) \in \mathcal{S}_+ (G,H)$ if only if, for all  $K\in \sum _{\mathcal{G}}(H)$,  the following conditions hold:   $D\ast K \in \mathcal{G}$ and  $ (D,\phi)\ast (\Delta(K),1) \in \mathcal{S}(D\ast K, K)$.
\item $(\mathcal{G}, \mathcal{S}_+)$ satisfies  the  axioms $(i)-(iv)$ in \ref{axioms}.
\item We  have   $\mathcal{D}\subseteq \mathcal{D}_+$, and equality holds if only if   $(\mathcal{G}, \mathcal{S})$ satisfies   axioms $(v)$ in  \ref{axioms}. In particular, by part $(b)$, we have $\mathcal{(D_+)_+ =D_+}$.
\item  Let $G, H \in \mathcal{G}$ and  $(D,\phi)\in \mathcal{M} (G, H)$ with  $p_1(D)=G$. Then $(D,\phi) \in \mathcal{S}(G,H)$ if only if $(D,\phi) \in \mathcal{S}_+(G,H)$. In particular, $\mathcal{D}_+$ contains the elementary operations  restriction, inflation, or deflation
if and only if the category  $\mathcal{D}$  does.
\end{enumerate}

\end{prop}
\begin{proof}

\begin{enumerate}[(a)]

\item  Let $G,H \in \mathcal{G}$ and $(D,\phi) \in \mathcal{M}(G\times H)$. First, assume  that $(D,\phi)\in \mathcal{S}_+ (G,H)$ and $K \in  \sum_{\mathcal{G}} (H)$. Then $p_1(D)\in \mathcal{G}$ and $(D,\phi) \in \mathcal{S}(p_1(D), H)$, applying   the Axiom  $(iv)$ of  $\mathcal{(G, S)}$ to $(D,\phi) \in \mathcal{S}(G,H)$ and $K$, we obtain $D\ast K \in \mathcal{G}$ and $(D\ast \Delta(K), \phi \ast 1 )\in \mathcal{S}(D\ast K, K)$. \\
For the converse. apply the condition  to $K=H$ and note that $D\ast H =p_1(D)$ and $(D,\phi) \ast (\Delta(H), 1)= (D,\phi) \in \mathcal{S} (p_1(D), H) $. Thus, we have $(D,\phi)\in \mathcal{S}_+ (G,H)$.
\item We know that $(\mathcal{G}, \mathcal{S}_+)$ satisfies   Axioms   $(i)$, $(ii)$ and  $(iii)$. Now we will prove the axiom $(iv)$. Let $G$, $H \in \mathcal{G}$ and  $(D, \phi)\in \mathcal{S}_+(G,H)$. Then we have   $p_1(D)\in \mathcal{G}$ and  $(D, \phi)\in \mathcal{S}(p_1(D),H)$. Now, for all  $K\in \sum_{\mathcal{G}}(H)$ and by axiom $(iv)$ of    $(\mathcal{G}, \mathcal{S})$, we have:
\begin{itemize}
\item $D\ast K =p_1(D\ast \Delta(K))\in \mathcal{G}$.
\item $(D\ast \Delta(K), \phi\ast 1 )\in \mathcal{S}(p_1(D\ast \Delta(K)), K)$.
\end{itemize}
Thus $(D\ast \Delta(K), \phi\ast 1 )\in \mathcal{S}_+(D\ast K, K)$.
\item Let $G, H \in \mathcal{G} $ and   $(D,\phi)\in \mathcal{S}(G, H)$, by axiom $(iv)$ of $(\mathcal{G, S})$,    $D\ast H=p_1(D)$ is an element of $ \mathcal{G}$ and $(D,\phi) \in \mathcal{S}(p_1(D), H)$. Thus  $\mathcal{S}(G, H)\subseteq \mathcal{S}_+(G, H)$, it follows that  
$\mathcal{D}\subseteq \mathcal{D}_+ $.\\
Additionally, if $\mathcal{D}= \mathcal{D}_+$, meaning that $\mathcal{S}(G,H)=\mathcal{S}_+(G,H)$ for any groups $G$ and $H$ in $\mathcal{G}$, and since $\mathcal{(G,S _+)}$ satisfies axiom $(v)$, then so does $\mathcal{(G,S)}$.

On the other hand, if $(\mathcal{G},\mathcal{S})$ satisfies axiom $(v)$, then for every $(D,\phi)\in \mathcal{S}_+(G,H)$, meaning that $p_1(D)\in \mathcal{G}$ and $(D,\phi)\in \mathcal{S}(p_1(D), H)$, and by Axiom $(v)$, we have $(\Delta(p_1(D)),1)\in \mathcal{S}(G,p_1(D))$. Therefore,
\begin{align*}
(\Delta(p_1(D)), 1)\ast (D,\phi) =(D,\phi)\in \mathcal{S}(G,H).
\end{align*}
Hence, $\mathcal{S}(G,H)+\subseteq \mathcal{S}(G,H)$, implying that $\mathcal{D}_+\subseteq \mathcal{D}$.
\item This follows immediately from the definition of $\mathcal{S}_+$.
 
\end{enumerate}
\end{proof}

\begin{defi}\label{S^+}
The data  $(\mathcal{G}, \mathcal{S})$ satisfies   Axioms   $(i)-(iii)$. For  $G, H\in \mathcal{G}$, we define  $\mathcal{S}^+(G,H)$  by 
\begin{align*}
\lbrace (D,\phi)\in \mathcal{M}(G \times H) \mid p_1(D) , p_2(D) \in \mathcal{G} \text{ and } (D,\phi)\in \mathcal{S}\left( p_1(D),p_2(D)\right)   \rbrace.
\end{align*}
\end{defi}
Note that  if   $(\mathcal{G},\mathcal{S})$ satisfies   Axioms $(vii) $ and $(viii)$,  then  $(\mathcal{G},\mathcal{S}^{+})$  satisfies  Axioms $(i)$, $(ii)$ and $(iii)$ (the same arguments as for  $\mathcal{S}_+$),  As  a result, we obtain a category  $  \mathcal{D^+ := C}^A\mathcal{(G, S^+)}$.

\begin{prop} \label{D^+}
The data   $(\mathcal{G}, \mathcal{S})$ satisfying Axioms  $(i)$-$(iii)$, additional  Axioms  $(vii)$ and $(viii)$, and   $\mathcal{G}$ is closed under interceptions.
\begin{enumerate}[(a)]
\item For  $G$, $H \in \mathcal{G}$, $(D,\phi)\in \mathcal{M}(G\times H)$ the following statements are equivalent:
\begin{enumerate}[(i)]
\item $(D,\phi) \in \mathcal{S}^+(G,H)$.
\item  $p_2(D)\in \mathcal{G}$ and  for all $L\in \sum_\mathcal{G} (p_2(D))$ one has  $D\ast L \in \mathcal{G}$ and  $(D,\phi)\ast (\Delta (L),1) \in \mathcal{S}(D\ast L,L)$ .
\item $p_1(D)\in \mathcal{G}$  and for all $K\in \sum_\mathcal{G} (p_1(D))$ one has $K\ast D \in \mathcal{G}$ and $(\Delta (K),1)\ast(D,\phi)  \in \mathcal{S}(K, K\ast D)$ .
\end{enumerate}
In particular,  $\mathcal{S} (G,H) \subseteq \mathcal{S}^+(G,H)$  for all  $G, H \in \mathcal{G}$.
\item $(\mathcal{G}, \mathcal{S}^+)$ satisfies Axioms  $(i)-(iii)$,  $(v)$, $(vi) $, $(vii)$ and   $(viii)$.
\item  One has  $\mathcal{D} \subseteq \mathcal{D}^+$,  with equality if only if  $(\mathcal{G}, \mathcal{S})$ satisfies Axioms $(v)$ and $(vi)$.  In particular, by  $(b)$, one has  $\mathcal{(D^+)^+=D^+}$.
\item  Let  $G, H \in \mathcal{G}$, and  $(D,\phi) \in \mathcal{M}(G\times H)$ such that  $p_1(D)=G $  and  $p_2(D)=H$, then $(D,\phi)\in \mathcal{S}(G,H)$ if only if $(D,\phi)\in \mathcal{S}^+(G,H)$. In particular, $\mathcal{D}^+$  contains a given  inflation or  deflation  if only if  $\mathcal{D}$ does.

\end{enumerate}
\begin{proof}
\begin{enumerate}[(a)]
\item 
\begin{itemize}
\item  $(i) \Rightarrow (ii)$.
Let $(D,\phi)\in \mathcal{S}^+(G,H)$, then $p_1(D), p_2(D)$ be elements of $ \mathcal{G}$ and  $(D,\phi)\in \mathcal{S}(p_1(D),p_2(D))$.   By Axiom  $(vii)$  of  $(\mathcal{G}, \mathcal{S)}$,  for any  $L\in \sum_\mathcal{G} (p_2(D))$, one has  $ D\ast L\in \mathcal{G}$ and  $(D,\phi) \ast (\Delta(L),1) \in \mathcal{S}(D\ast L, L) $. Thus $(ii)$ holds.
\item 
$(ii)\Rightarrow (i)$. Assume  $(D,\phi)$  satisfies the condition in $(ii)$.  Then  $p_2(D) \in \mathcal{G}$, $D\ast p_2(D)=p_1(D)\in \mathcal{G}$, moreover  $$(D,\phi) \ast (\Delta(p_2(D)), 1)=(D,\phi) \in \mathcal{S}(p_1(D), p_2(D)).$$ Thus  $ (D,\phi) \in \mathcal{S}^+(G,H)$.
\item Since Axioms $(vii)$ and $(viii)$ are symmetric, the equivalence of $(ii)$ and $(iii)$
is proved in the same way.
\end{itemize}
\item We proved  that $\mathcal{(G, S^+)}$ satisfies the axioms $(i)$-$(iii)$.
\begin{itemize}
 \item Axiom $(v)$. Let $G\in \mathcal{G}$  and let  $H\in \sum_{\mathcal{G}} G$. First note that   $(\Delta (H), 1) \in S(H,H)$ and  $p_i(\Delta(H))=H \in \mathcal{G}$  for   $i=1,2$, this means  that  $(\Delta (H),1) \in \mathcal{S}(p_1(H), p_2(H))$, moreover $(\Delta(H),1)\in \mathcal{M}(G\times H)$. Thus  $(\Delta(H),1)\in \mathcal{S}^+(G,H)$. 
 \item With the same argument, we have the Axiom $(vi)$
\item Axiom $(viii)$.  Let  $G , H \in \mathcal{G}$  and let  $(D,\phi)\in \mathcal{S}^+(G,H)$,   then  $p_i(D)\in \mathcal{G}$  for  $i=1,2$, and  $(D,\phi)\in \mathcal{S}(p_1(D),p_2(D))$. For any   $K\in \sum_{\mathcal{G}} p_1(D)$,  by Axiom $(viii)$ of  $\mathcal{(G,S)}$, one has  $K\ast D \in \mathcal{G}$, and  
 \begin{align*}
 (\Delta(K),1)\ast (D,\phi)\in \mathcal{S}(K, K\ast D).
 \end{align*}
Note that  $p_2(\Delta (K)\ast D)= K\ast D$ and  $p_1(\Delta(K)\ast D)=K$, Thus 
  \begin{align*}
 (D,\phi)\ast (\Delta(K),1)\in \mathcal{S}^+( K, K\ast D).
 \end{align*}
 \item Since Axioms $(vii)$  and $(viii)$ are symmetric, the Axioms $(vii)$ is proved en the same way. 
\end{itemize}
\item  If $\mathcal{D}=\mathcal{D}^+$. Then $(\mathcal{G,S})=(\mathcal{G,S^+})$,  by  $(b)$ one has  $(\mathcal{G, S})$  satisfies  Axioms $(v)$ and  $(vi)$. \\
On the other hand, if  $\mathcal{(G,S)}$ satisfies  Axioms  $(v)$ and  $(vi)$. Then,  for all  $G$, $H \in \mathcal{G}$  and for all  $(D,\phi)\in \mathcal{S}(G,H)$. By axioms $(vii)$ and $(vi)$, $p_2(D)\in \mathcal{G}$ and $(\Delta(p_2(D),1) \in \mathcal{S}(H,p_2(D)) $. Then 
\begin{align*}
(D,\phi) \ast (\Delta(p_2(D)), 1)=(D,\phi) \in\mathcal{S}(G,p_2(D))
\end{align*}
 Now, by Axioms $(viii)$ and  $(v)$ of $\mathcal{(G, S)}$, one has    $p_1(D)\in \mathcal{G}$ and $(\Delta(p_1(D),1)\in \mathcal{S}(p_1(D),G)$, hence  
 \begin{align*}
 (\Delta(p_1(D)),1)\ast (D,\phi)=(D,\phi) \in \mathcal{S}(p_1(D), p_2(D)).
 \end{align*}
Thus $(D,\phi)\in \mathcal{S}^+(G,H)$, it follows that  $\mathcal{D \subseteq D^+}$.\\
  Let  $(D,\phi)\in \mathcal{S}^+(G,H)$, then, $p_1(D), p_2(D) \in \mathcal{G}$ and $ (D, \phi) $ is an element of $ \mathcal{S}(p_1(D),p_2(D))$. Now, by axiom $(v)$ of   $\mathcal{(G,S)}$,  we have  $(\Delta(p_1(D)),1)\in \mathcal{S}(G,p_1(D))$ and by Axiom $(vi)$ of $(\mathcal{G,S})$,  we have $ \Delta(p_2(D)),1)\in \mathcal{S}(p_2(D),H)$, then 
\begin{align*}
(\Delta(p_1(D)),1)\ast (D,\phi)\ast (\Delta(p_2(D)),1)=(D,\phi) \in \mathcal{S}(G,H)
\end{align*}
it follows  $\mathcal{D^+ \subseteq D}$.
\item  It follows from the definition of $\mathcal{S}^+$.
 \end{enumerate}
\end{proof}
\end{prop}
\section{The functor  lower plus for $A$-fibered bisets} \label{sect functor lower }
Throughout this section, let $R$ denote a commutative ring and let the data $(G, S)$ satisfy Axioms $(i)–(iv)$. Define $\mathcal{D := C(G, S)}$ as in Definition \ref{S}, and define $\mathcal{S_+}$ and $\mathcal{D_+}$ as in Definition \ref{defi S_+}. The goal of this section is to construct a functor $-_+ : \mathcal{F}_{\mathcal{D},R}^{A} \longrightarrow \mathcal{F}^A_{D_+,R}$ that generalizes the construction in [\cite{constructios+}, Section 4].
Let  $X$ be an  $A$-fibered $G$-set,  we denote  by  $G_x$ the stabilizer  of the  $A$-orbit of  $x\in X$.
\begin{defi}
Let $G \in \mathcal{G}$, let $X $ be an $A$-fibered $G$-set such that $G_x \in \mathcal {G}$ for all
$x \in X$, and let  $F\in \mathcal{F}^A_{\mathcal{D},R}$.  A section of $F$ over $X$ is a function
\begin{align*}
s: X\longrightarrow \bigoplus_{x \in [X/A]} F(G_x)
\end{align*}
such that  $s(x)\in F(G_x)$  for all  $x\in X$ and $ [X/A]$ is  a  representative set of  the $A$-orbits  of $X$.
\end{defi}
These sections form a $R$-module via point-wise
constructions. The group $G\times A $ acts $R$-linearly on the set of these sections by  the action, $$(g,a)\cdot s (x)={^g s(g^{-1} ax)}=F(Iso(C_g))(s(g^{-1} ax)),$$ where $C_g: G_{{g^{-1}}x}\longrightarrow G_x$
is the conjugation isomorphism mapping. A section $s$ of $G$
over $F$ is called $(G,A)$-equivariant,  if $(g,a)\cdot s = s$ for all $(g,a) \in G\times A$.
\begin{defi} \label{GammaG}
Let $G\in \mathcal{G}$ and $F\in \mathcal{F}^A_{\mathcal{D},R}$. We will denote by $\Gamma_{F}(G)$ the category whose objects are pairs $(X,s)$ where $X$ is a $G$-set $A$-fibered such that $G_x\in \mathcal{G}$ for every $x\in X$, and $s$ is a $(G,A)$-invariant section of $F$ over $X$. Given $(X,s)$ and $(Y,t)$ objects in $\Gamma_{F}(G)$, an arrow $\alpha: (X,s) \longrightarrow (Y,t)$ is a morphism of $G$-sets $A$-fibered, $\alpha: X\longrightarrow Y$, satisfying $G_x=G_{\alpha(x)}$ and $t(\alpha(x))=s(x)$ for every $x\in X$. The composition in $\Gamma_F(G)$ is the composition of morphisms of $G$-sets $A$-fibered, and identity arrows are the identity morphisms.
\end{defi}
\subsection{ The functor $\Gamma_{F}(U): \Gamma_{F}(G) \longrightarrow \Gamma_{F}(H)$.}
 Let $G, H \in \mathcal{G}$, $F\in \mathcal{F}^A_{\mathcal{D}, R}$, and let  $U$ be an $A$-fibered $(G,H)$-biset  such that $((G\times H)_u, \phi_u)\in \mathcal{S}_+(G,H)$ for all  $u\in U$. 

\begin{defi} \label{GammaU}
Let  $U$ be an $A$-fibered  $(G,H)$-biset as above, we define a functor  
\begin{align*}
 \Gamma_{F}(U): \Gamma_{F}(H) &\longrightarrow \Gamma_{F}(G)\\
 (X,s)&\longrightarrow (U\otimes_{AH} X, U(s)),
 \end{align*}
 where
\begin{align*}
U(s)(u \otimes_{AH} x)=F\left( \left[ \frac{G_{u\otimes_{AH} x} \times H_x}{((G\times H_x)_u, \phi_{u,x})} \right]  \right) (s(x)),
\end{align*}
where  $((G\times H_x)_u, \phi_{u.x}):=((G\times H)_u), \phi_u) \ast (\Delta(H_x), 1)$.
For a morphism  \-$\alpha : (X,s)\longrightarrow (Y,t)$ in  $\Gamma_{F}(H)$
we set  $\Gamma _{F}(U)(\alpha)=U\otimes_{AH} \alpha$. 
\end{defi}
For the rest of
this subsection, we show that these definitions are well-defined and yield a functor.
\begin{proof}
First, we prove that $F$ can be applied to this class of the $A$-fibered biset. 
By definition of  \ref{GammaG}, one has $H_x\in \sum_{\mathcal{G}}(H)$ for all  $x\in X$,  and by hypothesis over  $U$ one has $((G\times H)_u, \phi_u)\in \mathcal{S}_+(G, H)$, then $ p_1((G\times H)_u)\in \mathcal{G}$ and  by axiom $(iv)$ of $\mathcal{(G,S)}$:
 \begin{itemize}
 \item $(G\times H)_u \ast H_x= G_{u\otimes x}\in \mathcal{G}$.
 \item $((G\times H)_u, \phi_u)\ast (\Delta(H_x),1)=((G\times H_x)_u, \phi_{u,x})\in \mathcal{S}(G_{u\otimes x}, H_x)$.
 \end{itemize}
Therefore, $F$ can be applied to the class of this  $A$-fibered biset.\\
 Next, we will proved that the definition $U(s)(u\otimes_{AH} x)$ remains unchanged when replacing $u\otimes_{AH} x$ with $ua^{-1} h^{-1} \otimes_{AH} hax$ for some $h \in H$ and $a \in A$.
\begin{align*}
G_{ua^{-1} h^{-1} \otimes_{AH} hax}&= G_{u\otimes_{AH} x},\\
 H_{hax}= {^hH_{ax}}=&{^hH_x}=H_{hx},\\
 (G\times H)_{ua^{-1}h^{-1}}&={^h (G\times H)_{ua^{-1}}} ={^h(G\times H)_u}.
\end{align*}
Then 
\begin{align*}
U(s)(u^{-1}h^{-1} \otimes hax)&=F\left( \left[ \frac{G_{ua^{-1}h^{-1}\otimes_{AH} hax} \times H_{hax}}{((G\times H_{hax})_{u(ha)^{-1}}, \phi_{ua^{-1}h^{-1},hax})} \right]  \right) (s(hax))
\end{align*}
Now, note that 
\begin{align*}
(G\times H_{hax})_{u(ha)^{-1}} &=(G\times H)_{ua^{-1}h^{-1}}\ast \Delta (H_{hax})\\
&=(G\times H)_{uh^{-1}}\ast \Delta (H_hx)\\
&={^{(1,h)}(G\times H)_u} \ast {^{(h,h)}\Delta (H_x)} 
\end{align*}
this is equal to  $(G\times H)_u \ast \Delta (H_x)\ast \lbrace (h^{-1} h^\prime h, h^\prime )\mid h^\prime \in H_{hx} \rbrace$. Moreover,  if  $(g^\prime, h^\prime)\in (G\times H_{hax})_{u(ha)^{-1}}$, one has 
\begin{align*}
\phi_{ua^{-1}h^{-1},hax}(g^\prime, h^\prime)&= \phi_{ua^{-1}h^{-1}} (g^\prime, h^\prime)\cdot 1((h^\prime, h^\prime))\\
&= \phi_{u(ha)^{-1}}(g^\prime, h^\prime). 
\end{align*}
On the other hand 
\begin{align*}
\phi_{u.x} \ast 1 ((g^\prime. h^\prime))&=\phi_{u,x} ((g^\prime, h^{-1} h^\prime h))\cdot 1(( h^{-1} h^\prime h, h^\prime))\\
&=\phi_{u,x}( (g^\prime, h^{-1} h^\prime h))\\
&=\phi_u((g^\prime, h^{-1} h^\prime h)).
\end{align*}
Now let's note that since $(g^\prime, h^{-1}h^\prime h) \in (G \times H)_u$, if $a^\prime = \phi_u(g^\prime, h^{-1}h^\prime h)$, we have that $(g^\prime, h^{-1}h^\prime h) \cdot u = a^\prime_u$, then $$a^\prime uh^{-1} = g^\prime uh^{-1} (h^\prime)^{-1}h h^{-1} = g^\prime (uh^{-1}) (h^\prime)^{-1},$$ and $a^\prime = \phi_{uh^{-1}} (g^\prime, h^\prime)$.  
  Then  $\phi_{ua^{-1}h^{-1},hax} = \phi_{u.x} \ast 1$. In consequence:
\begin{align*}
U(s)(u^{-1}h^{-1} \otimes _{AH} hax)&= F\left( \left[ \frac{G_{u\otimes_{AH} x} \times H_x}{((G\times H_x)_u, \phi_{u,x})} \right]  \otimes_{AH_x}             \left[ \frac{H_x\times H_{hx}}{(\Delta_{c_h}(H_{hx}), 1)}\right]   \right) (s(hax))\\
&= F\left( \left[ \frac{G_{u\otimes_{AH} x} \times H_x}{((G\times H_x)_u, \phi_{u,x})} \right] \right)
 (^hs(ax))\\
&= F\left( \left[ \frac{G_{u\otimes_{AH} x} \times H_x}{((G\times H_x)_u, \phi_{u,x})} \right] \right)  s(x)\\
&=U(s)(u\otimes_{AH} x).
\end{align*}
Thus $U(s)$ is a section of $U\otimes X$ over $F$.\\
Next,  we show that the section $U(s)$ is $(G,A)$-equivariant. For any $g \in G$ and $a\in A$, we have
\begin{align*}
(g,a)\cdot U(s)(u\otimes x)&={^gU(s)(g^{-1}au\otimes x )}\\
&=F\left( \left[ \frac{G_{u\otimes x }\times G_{g^{-1}au\otimes x}}{(^{(g,1)}\Delta(G_{g^{-1} u \otimes x}),1)} \right] \otimes_{A G_{g^{-1}au\otimes x}} \left[ \frac{G_{g^{-1}au\otimes x} \times H_x}{((G\times H)_{g^{-1}au}, \phi_{g^{-1}a u, x} )} \right]  \right) (s(x))\\
&=F\left(  \left[ \frac{G_{u\otimes x} \times H_x}{((G\times H_x)_{u}, \phi_{u, x} )} \right]  \right) (s(x)),
\end{align*} 
since $^{(g,1)}\Delta (G_{g^{-1} u \otimes x })\ast (G\times H_x)_{g^{-1} u}= (G\times H_x)_{u} $, and for any
$(g^{\prime},h^{\prime})$ in $(G\times H_x)_u$,  we have 
\begin{align*}
1\ast\phi_{g^{-1}a u, x} (g^\prime, h^\prime))& = 1((g^\prime, g^{-1}g^\prime g))\cdot \phi_{g^{-1}a u}((g^{-1}g^\prime g, h^\prime))\\
&=\phi_{g^{-1}a u}((g^{-1}g^\prime g, h^\prime)).
\end{align*}
Given $(g^{\prime},h^{\prime})\in (G\times H_x)u$, we have that
\begin{align*}
1\ast\phi{g^{-1}a u, x} ((g^\prime, h^\prime))
&= 1((g^{\prime}, g^{-1} g^\prime g) ) \cdot \phi_{g^{-1}a u, x}((g^{-1}g^\prime g, h^\prime))\\
&=\phi_{g^{-1}a u}((g^{-1}g^\prime g, h^\prime)),
\end{align*}
where
\begin{align*}
(g^{-1}g^\prime g, h^\prime)\cdot g^{-1} u = g^{-1}g^\prime g (g^{-1} u) (h^\prime)^{-1} = g^{-1}a^{\prime} u =a^{\prime} g^{-1}u
\end{align*}
and $a^{\prime}=\phi_{u}((g^{\prime},h^{\prime}))$. Therefore, $1\ast\phi_{g^{-1}a u, x} =\phi_{u,x}$. Thus, 
\begin{align*}
 (g,a)\cdot U(s)(u\otimes x)&={^gU}(s)(g^{-1}au\otimes x )\\
&=F\left(  \left[ \frac{G_{u\otimes x} \times H_x}{((G\times H_x)_{u}, \phi_{u, x} )} \right]  \right) (s(x)).
\end{align*} 
Now, we will proved that $\Gamma_{F}(U)=U\otimes{{AH}} \alpha$ is a morphism from $(U\otimes{AH} X, U(s))$ to $(U\otimes_{AH} Y, U(t))$ in $\Gamma_{F}(G)$. 
Since $H_x=H_{\alpha(x)}$, one has  
\begin{align*}
G_{u\otimes \alpha(x)}=(G\times H)_u \ast H_{\alpha(x)}=  (G\times H)_u \ast H_x=G_{u\otimes x},
\end{align*}
for any $u\in U$ and $x\in X$. Thus,
\begin{align*}
U(t)(u\otimes \alpha (x))&=F \left( \left[ \frac{G_{u\otimes \alpha(x)} \times H_{\alpha(x)}}{(G\times H_{\alpha(x)})_u, \phi_{u,\alpha(x)} } \right]  \right) (t(\alpha (x)))\\
&=F\left(  \left[ \frac{G_{u\otimes x} \times H_x} {(G\times H_x)_u, \phi_{u,x}}\right] \right) (s(x)) \\
&=U(s)(u\otimes x),
\end{align*}
then,  $U \otimes_{AH} \alpha$ is a morphism in $\Gamma_F (G)$.\\
Additionally, $\Gamma_{F}(U)$ preserves composition and identity since $U\otimes {-}$ does. Therefore, the functor $\Gamma{F}(U)$ is well-defined.
\end{proof}
\begin{prop} \label{Gammacomposicion}
Let  $G$, $H$ and  $K$ be finite groups in $\mathcal{G}$, let  $U$ be an $A$-fibered  $(G, H)$-biset such that  $((G\times H)_u, \phi_u)\in \mathcal{S}_+(G.H)$ for all $u\in U$ and let  $V$  be  an $A$-fibered $(H, K)$-biset  such that  $((H\times K)_v, \phi_v)\in \mathcal{S}_+(H, K)$ for all  $v\in V$. Moreover, the two functors  $\Gamma_{F}(U\otimes_{AH} V)$ and  $\Gamma_F(U) \circ \Gamma_{F}(V)$ are naturally isomorphic.
\end{prop}
\begin{proof}
Let  $u \otimes v \in  U \otimes_{AH} V$.  By lemma \ref{estProd},  
\begin{align*} 
((G\times K)_{u\otimes v}, \phi_{u\otimes v})=((G\times H)_u,\phi_u) \ast ((H\times K)_v, \phi_v)
\end{align*}
Thus  $((G\times K)_{u\otimes v}, \phi_{u\otimes v}) \in \mathcal{S}_+ (G,K)$.
Let  $(X,s)$ be an object of  $\Gamma_{F} (K)$.  Now, we prove that  $( (U\otimes _{AH} V) \otimes_{AK} X, (U\otimes V)(s))$ and $(U\otimes_{AH}(V\otimes _{AK} X), U(V(s)))$ are isomorphic in  $\Gamma_{F} (G)$.
We set 
\begin{align*}
\alpha: (U\otimes_{AH} V)\otimes_{AK} X &\longrightarrow U\otimes_{AH} (V\otimes_{AK} X )\\
(u \otimes v)\otimes x &\longmapsto u\otimes (v\otimes x)
\end{align*}
is an isomorphism in  $\Gamma_{F} (G)$, since $\alpha$ is an isomorphism of $A$-fibered bisets and  $G_{(u\otimes v)\otimes x}= G_{u\otimes (v\otimes x)}$. Then, one has 
\begin{align*}
(U\otimes_{AH} V)(s)((u\otimes v)\otimes x)=F \left( \left[ \frac{G_{(u\otimes v)\otimes x \times K_x}}{(G\times K_x)_{(u\otimes v) } , \phi_{(u\otimes v), x}}\right]  \right) (s(x)).
\end{align*}
On the other hand
\begin{align*}
U(V(s))(u\otimes (v\otimes x))&=F\left(\left[\frac{G_{u\otimes (v\otimes x)} \times H_{v\otimes   x}}{(G\times H_{v\otimes x} )_u , \phi_{u,v\otimes x} } \right]   \right) (V(s)(x))\\
&= F\left(\left[\frac{G_{u\otimes (v\otimes x)} \times H_{v\otimes   x}}{(G\times H_{v\otimes x} )_u , \phi_{u,v\otimes x} } \right] \otimes_{AH_{v\otimes x}} \left[ \frac{H_{v\otimes x} \times K_x}{(H\times K_x)_v, \phi_{v,x}} \right]    \right) (s(x)).
\end{align*}
Moreover, the following equalities hold:
\begin{itemize}
\item $(H\times K_x)_v=(H\times K)_v \ast \Delta(K_x)$,
\item $H_{v\otimes x}=(H\times K)_v \ast K_x,$
\end{itemize}
hence, $p_1((H\times K_x)_v)= H_{v\otimes x}$. By the Mackey formula, we have:
\begin{align*}
U(V(s))(u\otimes (v\otimes x))&= F\left(\left[\frac{G_{u\otimes (v\otimes x)} \times K_{x}}{((G\times H_{v\otimes x} )_u , \phi_{u,v\otimes x}) \ast ((H\times K_x)_v, \phi_{v,x}) } \right] \right) (s(x)).
\end{align*}
Recall:
\begin{itemize}
\item $(G\times K_x)_{u\otimes v} = (G\times K)_{u\otimes v} \ast \Delta (K_x)= (G\times H)_u \ast (H\times K)_v \ast \Delta(K_x)$.
\item $(G\times H_{v\otimes x})_{u}= (G\times H){u} \ast \Delta (H_{v\otimes x})$.
\item
\begin{align*}
(G\times H_{v\otimes x})_{u} \ast  (H\times K_x)_v & =(G\times H)_{u} \ast \Delta (H_{v\otimes x}) \ast (H\times K_x)_v\\& = (G\times H)_{u} \ast \Delta (p_1((H\times K_x)_v  ) )\ast (H\times K_x)_v \\
&=(G\times H)_u\ast (H\times K_x)_v\\ &=(G\times H)_{u} \ast (H\times K)_v \ast \Delta(K_x)\\ &= (G\times K_x)_{u\otimes v} .
\end{align*}
\begin{align*}
\phi_{u\otimes v,x} = \phi_{u\otimes v} \ast 1 =\phi_{u\otimes v} = \phi_u \ast \phi_v,
\end{align*} 
on the other hand,
\begin{align*}
\phi_{u, v\ast x} \ast \phi_{v,x}= (\phi_u \ast 1) \ast (\phi_v \ast 1)= \phi_u \ast \phi_v.
\end{align*} 
\end{itemize}
Therefore, $U(V(s))= (U\otimes_{AH} V)(s)$.
\end{proof}

\begin{defi}
For   $G\in \mathcal{G}$ and  $F\in \mathcal{F}_{\mathcal{D},R}^{A}$, we define the following two operations in the category  $\Gamma_{F}(G)$.
\begin{enumerate}[(a)]
\item  Let  $(X, s)$ and  $(Y, t)$ be elements of  $\Gamma_{F}(G)$, their co-product  $(X, s)\sqcup (Y,t)$  is defined as  $(X \sqcup Y, s\sqcup t)$, where $X\sqcup Y$ is the disjoint union of $X$ and  $Y$ and  $s\sqcup t$ is the section on  $X\sqcup Y$  which is defined on  $X$ as  $s$ and  on $Y$ as $t$. This construction  together with the obvious inclusions from $X$ and  $Y$  to $X\sqcup Y$ is also a categorical coproduct   $\Gamma_{F}(G)$.
\item  For $(X, s)$ and  $(X, t)$  elements of $\Gamma_F(G)$, with the same $A$-fibered $G$-set $X$,  we have an object  $(X, s + t)$, where $s + t$  is the pointwise  sum of  $s$ and  $t$.
\end{enumerate}
\end{defi}

\begin{defi}
Let $G\in \mathcal{G}$ and  $F\in \mathcal{F}^A_{\mathcal{D},  R}$. We define the abelian group $F_+(G)$  as the free abelian group on the set  of isomorphism classes   $\lbrace X, s\rbrace$ of objects $(X, s)$ of  $\Gamma_{F}(G)$ module the subgroup generated by all  elements of the form
\begin{equation}\label{cociente}
\lbrace X\sqcup Y, s \sqcup t\rbrace-\lbrace X, s\rbrace-\lbrace Y,t\rbrace \ \text{ and  }  \ \lbrace X, s+r\rbrace-\lbrace X,s\rbrace-\lbrace X,r\rbrace
\end{equation}
where $(X, s)$, $(X, r)$, $(Y, t)$  are objects  of  $\Gamma_{F}(G)$. We will denote  the  coset  of $\lbrace X, s\rbrace$ in  $F_+(G)$ by $[X, s]_G$.
\end{defi}
Let's note that, due to the first relation, the classes of elements $(\frac{G\times A}{H,\phi} ,s)$, where $H \in\sum_\mathcal{G}(G)$, $\phi: H\longrightarrow A$, and $s$ is a $(G,A)$-invariant section of $F$ over the $G$-set $A$-fibered $\frac{G\times A}{H,\phi}$, form a set of generators for the abelian group $F_+(G)$.

For any $x\in F(H)$, there exists a unique $(G,A)$-invariant section $s_x$ of $F$ over $\frac{G\times A}{H,\phi}$ with $s_x((H,\phi)) = x$. The class $\lbrace \frac{G\times A}{H,\phi}, s_x\rbrace$ will be denoted by $[H,\phi, x]G \in F_+(G)$.
\begin{thm}
Let $R$ be a commutative ring and  $(\mathcal{G}, \mathcal{S})$ satisfying Axioms  $(i)$ to  $(iv)$,  set  $\mathcal{D}:=\mathcal{C} (\mathcal{G}, \mathcal{S} )$ and  $\mathcal{D}_+:=\mathcal{C} (\mathcal{G}, \mathcal{S}_+ )$ and  let  $F\in \mathcal{F}^A_{\mathcal{D}, R}$ be a biset functor  on  $\mathcal{D}$ over  $R$.
\begin{enumerate}[(a)]
\item Mapping a finite group $G$  to $R$-module $F_+(G)$ and an element  $[U] \in B^A(G,H)$, where  $U$ is a  $A$-fibered $(G,H)$-biset with  stabilizers pairs  in  $\mathcal{S}_+ (G,H)$ to  $R$-lineal map  $F_+([U]):F_+(H)\longrightarrow F_+(G)$ induced by the functor  $\Gamma_F([U]):\Gamma_F(H)\longrightarrow \Gamma_F(G)$, yields a biset functor $F_+\in \mathcal{F}_{\mathcal{D}_+,R}^A$.
\item The association  $F \longrightarrow F_+$  define a  $R$-lineal functor $$\textbf{-}_+: \mathcal{F}_{\mathcal{D}, R}^A\longrightarrow \mathcal{F}_{\mathcal{D}_+, R}^A.$$

\end{enumerate}
\end{thm}

\begin{proof}
\begin{enumerate}[(a)]

\item It is straightforward to show that if $U$ and $V$ are isomorphic  $A$- fibered $(G, H)$-bisets
with pairs stabilizers in $S_+(G, H)$, then the functors $\Gamma_F (U)$ and $\Gamma_F (V )$ are naturally isomorphic. \\
Let $\rho: U \longrightarrow V$ be an isomorphism of $(G,H)$-biset. According to Definition \ref{GammaU}, we only need to note that for every $H$-set $A$-fibered $X$:
\begin{itemize}
    \item The morphism $\rho \otimes_{AH} 1: U\otimes_{AH} X \longrightarrow V\otimes_{AH} X$ is an isomorphism of $G$-sets. Therefore, $G_{u\otimes_{AH}x }= G_{\rho(u)\otimes_{AH}x }$.
\item For any element $u\otimes_A x\in U\otimes_A X$, we have:
\begin{align*}
U(s)(u\otimes_{AH} x)&= F\left( \left[ \frac{G_{u\otimes_{AH} x} \times H_x}{((G\times H_x)_u, \phi_{u,x})} \right]  \right) (s(x))\\
&=F\left( \left[ \frac{G_{\rho(u)\otimes_{AH} x} \times H_x}{((G\times H_x)_{\rho(u)}, \phi_{\rho(u),x})} \right]  \right) (s(x))\\
&= V(s) (\rho(u)\otimes_{AH} x)
\end{align*}
The last equality follows from the fact that:
\begin{align*}
((G\times H_x)_u, \phi_{u.x})&=((G\times H)_u), \phi_u) \ast (\Delta(H_x), 1))\\
&=((G\times H)_{\rho (u)}, \phi_{\rho (u)} ) \ast (\Delta(H_x), 1)\\
&= ((G\times H_x)_{\rho (u)}, \phi_{\rho(u).x} ).
\end{align*}
\item  Given a morphism $\alpha : (X,s)\longrightarrow (Y,t)$ in $\Gamma_{F}(H)$. By definition, we have $\Gamma _{F}(U)(\alpha) =U\otimes_{AH} \alpha$, then:

\end{itemize}
Therefore, $\Gamma_F (U)$ and $\Gamma_F (V)$ are naturally isomorphic. The functor $\Gamma_F (U)$ can be extended to a well-defined group homomorphism $f_{[u]}$, from the free abelian group associated to $\Gamma_F(H)$ to the free abelian group associated to $\Gamma_F(G)$. Additionally, note that $$f_{[U]} \left( (X\sqcup Y, s \sqcup t) -(X, s)- (Y,t)\right)$$ is equal to
\begin{align*}
  (U\otimes_{AH} (X\sqcup Y),U( s \sqcup t)) -(U\otimes_{AH} X, U(s))- (U\otimes_{AH} Y,U(t)),
\end{align*}
which is isomorphic to
$$ ((U\otimes_{AH} X)\sqcup ( U\otimes_{AH} Y), U(s)\sqcup U(t))-(U\otimes_{AH} X, U(s))- (U\otimes_{AH} Y,U(t)).$$ 
Thus, $f_{[U]}$ induces a homomorphism of $R$-modules 
 $$F_+([U]): F_+(H) \longrightarrow F_+(G),$$
$f_{[U\sqcup V]}= f_{[U]}+ f_{[V]}$. Therefore, we can define a group homomorphism $F_+: B^A(G,H)\longrightarrow Hom_{R}(F_+(H), F_+(G))$. Furthermore, by Proposition \ref{Gammacomposicion}, it follows that $F_+$ belongs to $\mathcal{F}_{\mathcal{D}+,R}^A$.
\item Let  $F, F^\prime \in \mathcal{F}^A_{\mathcal{D},  R}$ and let  $\varphi: F\longrightarrow F^\prime$ be a morphism  in $\mathcal{F}_{D,R}^A$, then for every $G \in \mathcal{G}$, one has one induced  functor  $\Gamma_F (G)\longrightarrow \Gamma_{F^\prime} (G)$, which maps an object  $(X,s)\in \Gamma_F(G)$ to  $(X, \varphi(s))$, where  $\varphi(s)(x):=\varphi_{G_x}(s(x))$. This functor induced a  $R$-lineal map from $F_+(G)$ to $F^\prime_+(G)$. Now, we prove that this construction
respects compositions. Let  $F^{\prime \prime} \in \mathcal{F}_{\mathcal{D},R}^{A}$
and  let  $ \tau: F^\prime \longrightarrow F^{\prime \prime}$  be a morphism in $\mathcal{F}_{D,R}^A$, then  
\begin{align*}
(\tau \circ \phi)(s)(x) &=  (\tau \circ \phi)_{G_x}(s(x))=\tau_{G_x} \left( \phi_{G_x} (s(x)) \right) \\
&=\tau_{G_x}(\phi(s)(x))=\tau(\phi(s))(x).
\end{align*}
It is straightforward to verify that this construction
 sends the identity morphism to the identity morphism, and is
$R$-linear.
\end{enumerate}
\end{proof}


Let  $G, \ H$  be finite  groups of $\mathcal{G}$ and let  $U$ be an $A$-fibered  $(G,H)$-biset such that all stabilizer pairs are in  $S_+(G,H)$. If  $K\in \sum_{\mathcal{G}} (H)$  and  $a\in F(K)$,  one has 
\begin{align*}
F_+([U])([K, \psi, a]_H)&= F_+([U]) \left( \left\lbrace \dfrac{H\times A}{(K,\psi)} , s_a\right\rbrace \right)\\ &= \left\lbrace U\otimes_{AH} \dfrac{H\times A}{(K,\psi)} , U(s_a)\right\rbrace 
\end{align*}
by  (\ref{cociente}), this is equal to 
\begin{align*}
\sum_{ (u\otimes hK) \in [G\backslash U\otimes_{AH} (H\times A/(K,\psi)) / A]} \left\lbrace \left[ \dfrac{G\times H_{hK}}{  G_{[u\otimes_{AH}hK]},  \phi_{u\otimes_{AH} hK} }\right] , U(s_a) \right\rbrace  
\end{align*}
where  $[G\backslash U\otimes_{AH} (H\times A/K,\psi) / A]$ denotes a representative set of  the $(G,A)$-orbits of  $U\otimes_{AH} H\times A/(K,\psi)$. Now, note that 
\begin{itemize}
\item $H_{hK}={^hK}$.
\item   By Lemma \ref{estProd}, $\phi_{[u\otimes_{AH} hK]}=\phi_u \ast {^h\psi}$, 
\item By definition  
\begin{align*}
U(s_a)(u\otimes_{AH} hK)&=F\left( \left[ \frac{G_{u\otimes_{AH} hK} \times ^hK}{((G\times ^hK)_U, \phi_{u,hK}}\right] \right) s_a(hK)\\
&=F\left( \left[ \frac{G_{u\otimes_{AH} hK} \times ^hK}{(G\times ^hK)_U, \phi_{u, hK} }\right] \right)  (^ha). 
\end{align*}
\end{itemize}
Thus, $F_+([U])([K, \psi, a]_H)$ is equal to 
\begin{align*}
\sum_{ (u\otimes hK) \in [G\backslash U\otimes_{AH} \frac{H\times A}{K, \psi} / A]} [G_{u\otimes hk}, \phi_u\ast {^h\psi}, F\left( \dfrac{G_{u\otimes hK }\times ^hK}{(G\times ^hK)_u, \phi_{u,hK} }\right) (^h a) ]_G.
\end{align*}
In particular, if $U=(G\times H/D,\phi) $, one has  $F_+\left(\left[ G\times H/D,\phi\right]  \right) ([K,\psi, a])$ is equal to 
\begin{align} \label{F_+[trans](---)}
\sum_{\substack{ h\in [p_2(D)\backslash H/K]\\ \phi_2 \vert_ {X_h}={^h\psi_2\vert_{X_h}} }}\left[  D\ast {^hK}, \phi \ast {^h \psi},F  \left( \left[ \dfrac{ D\ast {^hK} \times ^h K}{(D\ast \Delta (^hK), \phi)}\right]  \right) (^ha )  \right]_G 
\end{align}
where  $[p_2(D)\backslash H/K]$ denotes a set of representative of the double of $H$ respect to $p_2(D)$ and $K$, and  $X_h:= k_2(D) \cap ^hK$. This formula specializes in the case of elements $A$-fibered biset as follows:   
\begin{itemize}
\item Let  $G\in \mathcal{G}$,   $H\in \sum_\mathcal{G} (G)$,   $(K, \psi)\in  \mathcal{M}^\prime (H)$ and  $a\in F(K)$, then
\begin{align*}
F_+(Res_H^G)([K, \psi, a]_G)= \sum_{\substack{ g\in [H\backslash G /H]\\ ^g\psi \vert_{H\cap ^gK} =1 } } [H\cap ^g K,1\ast  ^g\psi, F(Res_{H\cap ^g K}^{^gK}) (^g a)]_H.
\end{align*}
\item Let  $G\in \mathcal{G}$, $H\in \sum_\mathcal{G} (G)$ with  $(\Delta(H),1)\in S_+(H,G)$, $K\in \sum_\mathcal{G}(H)$, $\psi: K\longrightarrow A$ be a group homomorphism  and let  $a\in F(K)$. Then
\begin{align*}
F_+(Ind^G_H)([K,\psi, a]_H)=[K, \psi, a]_G.
\end{align*}
\item  Let  $G\in \mathcal{G}$,and  $N\unlhd G$ such that  $G/N \in \mathcal{G}$ and   $^{(1,\pi)}\Delta(G):=\lbrace (g, gN)\mid g\in G\rbrace \in \mathcal{S}_+(G,G/N)$. For  $N\lhd K \leq G$  with  $(K/N, \psi) \in \mathcal{M}^\prime (G/N)$ and $a\in F_+(K/N)$, one has 
\begin{align*}
F_+(Inf_{G/N}^G)([K/N, \psi, a]_{G/N})=[K,\overline{\psi},F(Inf^K_{K/N})(a)]_G
\end{align*}
where $\overline{\psi} :K\longrightarrow A$ is the homomorphism  induced by  $\psi$.

\item  Let  $G\in \mathcal{G}$,  and $N\unlhd G$ such that  $G/N \in \mathcal{G}$  and $^{(\pi,1)}\Delta(G):=\lbrace (g, gN)\mid g\in G\rbrace\in \mathcal{S}_+(G,G/N)$. For  $K \in\sum_{\mathcal{G}}(G)$ and  $a\in F(K)$, one has 
\begin{align*}
F_+(Def^G_{G/N})([K,\psi, a]_G)&=\left[  KN/N,1 \ast \psi, F\left(\left[ \dfrac{KN/N \times K}{ ^{(\pi,1)} \Delta (K), 1} \right] \right)  (a) \right] _{G/N}.
\end{align*}
 
\end{itemize}
\subsection{Multiplicative Structure of $F_+$}

Let $ (\mathcal{G}, \mathcal{S}) $ satisfy axioms $(i)$ to $(iv)$. We can define the category $ \mathcal{D} := \mathcal{C}^A(\mathcal{G}, \mathcal{S}) $, and let $ F \in \mathcal{F}_{\mathcal{D},R}^{A} $ such that $ F(G) $ has a multiplicative structure for every $ G \in \mathcal{G} $. The group $ F_+(G) $ can be given a multiplicative structure defined by
 $$[X, s]_G \cdot [Y, t]_G := [X \otimes Y, s \otimes t]_G,$$ where
\begin{align*}
s \otimes t : X\otimes Y &\longrightarrow \coprod_{x\otimes y \in [ (X\otimes Y) /A] } F(G_{x\otimes y})\\
x\otimes y &\longmapsto F(Res^{G_x}_{G_{x\otimes y}}) (s(x)) \cdot F(Res^{G_y}_{G_{x\otimes y}}) (t(y)),
\end{align*} 
with $(X, s)$, $(Y, t) \in \Gamma_F(G)$.

Translating the definition to basic elements, we have that $ [H, \phi, x]G \cdot [K, \psi, y]G $ is equal to
\begin{align*}
\sum_{g \in [H\backslash G/K]} [H\cap {^gK}, \phi\cdot {^g\psi}, F(Res^H{H\cap ^gK})(x) \cdot F(Res^{^gK}_{H\cap ^gK})({^gy})]_G.
\end{align*}

Now, we will show that this definition makes sense:
\begin{itemize}
\item First, let's proved that $s \otimes t$ is well-defined. For $a \in A$, we have that $s \otimes t(xa^{-1} \otimes ay)$ is equal to
\begin{align*}
&F(Res^{G_{xa^{-1}}}_{G_{{xa^{-1}}\otimes hay}})(s(xa^{-1}))\cdot F(Res^{G_{ay}}_{G_{x{a^{-1}}\otimes ay}})(t(ay))\\  
&=F(Res^{G_{x}}_{G_{x\otimes y}})(s(x{a^{-1}}))\cdot F(Res^{G_{y}}_{G_{x\otimes y}})(t(ay))\\
&=F(Res^{G_{x}}_{G_{x\otimes y}})(s(x))\cdot F(Res^{G_{y}}_{G_{x\otimes y}})(t(y))\\
&=s\otimes t(x\otimes y).
\end{align*} 
\item Since $s$ and $t$ are $A$-invariant sections, this definition does not depend on the set of representatives $[(X\otimes Y)/A]$.
\item The product does not depend on the representatives of the class. Let $(X, s)$, $(X', s')$, $(Y, t)$, $(Y', t')$ be in $\Gamma_F(G)$ such that $\lbrace X, s\rbrace=\lbrace X^\prime, s^\prime \rbrace$ and $\lbrace Y, t \rbrace =\lbrace Y^\prime, t^\prime\rbrace$; then there exist morphisms $\alpha:(X,s) \longrightarrow (X', s')$ and $\beta:(Y,t) \longrightarrow (Y', t')$ that are isomorphisms in $\Gamma_F(G)$, i.e., $G_x=G_{\alpha(x)}$, $G_y=G_{\beta(y)}$, $s(x)=s'(\alpha(x))$, and $t(y)=t'(\beta(y))$ for all $x\in X$ and $y\in Y$.

We define
\begin{align*}
\rho: X\otimes Y&\longmapsto X^{\prime} \otimes Y^\prime\\
x\otimes y &\longrightarrow  \alpha(x)\otimes \beta(y),
\end{align*} 
noting that $\rho$ is bijective, since $\alpha$ and $\beta$ are. Now we will proved that $\rho$ is well-defined. Let $x\in X$, $y\in Y$, and $a\in A$; we have
\begin{align*}
\rho (xa^{-1}\otimes ay)&= \alpha (xa^{-1})\otimes \beta(ax)\\&=\alpha(x)a^{-1}\otimes a\beta(y)\\&=\alpha(x)\otimes \beta(y)\\&=\rho(x\otimes y).
\end{align*}
To proved that $\rho$ is a morphism of $G$-sets $A$-fibered, we have that for all $g\in G$ and $a\in A$:
\begin{align*}
\rho((g,a)\cdot (x\otimes y))&=\rho(gax\otimes gy)\\&=\alpha(gax)\otimes \beta(gy)\\&=ga\alpha(x)\otimes g \beta(y)\\&=(g,a)\cdot \rho(x\otimes y).
\end{align*} 
Let  $x\in X$, $y\in Y$, one has  
\begin{align*}
(s^\prime \otimes t^\prime) (\rho (x\otimes y))&= (s^\prime \otimes t^\prime)(\alpha (x)\otimes \beta(y))\\&=F(Res^{G_{\alpha(x)}}_{G_{\alpha(x)\otimes \beta(y)}} ) (s^\prime(\alpha(x)) \cdot F(Res^{G_{\beta(y)}}_{G_{\alpha(x)\otimes \beta(y)}} ) ( t^\prime(\beta(y) )\\&=F(Res^{G_{x}}_{G_{x\otimes y}} ) (s^\prime(\alpha(x)) \cdot F(Res^{G_{y}}_{G_{x\otimes y}} ) ( t^\prime(\beta(y) )\\
&=F(Res^{G_{x}}_{G_{x\otimes y}} ) (s(x)) \cdot F(Res^{G_{y}}_{G_{x\otimes y}} ) ( t(y) )\\
&=(s\otimes t) (x\otimes y).
\end{align*}
Therefore, $\rho$ is an isomorphism in $\Gamma_F(G)$, then $[X\otimes Y, s\otimes t]= [X^\prime \otimes Y^\prime, s^\prime \otimes t^\prime]$.
\item Now we will show that the product is zero over the quotient of $F_+(G)$.
Let $(X,s)$, $(Y,f)$, $(Y,t)$, $(Z,r)\in \Gamma_F(G)$.
\begin{itemize}
\item 
\begin{align*}
(X,s)\cdot (Y\sqcup Z, t\sqcup r)&=(X\otimes (Y\sqcup Z), s\otimes (t\sqcup r) )\\&= ((X\otimes Y)\sqcup (X\otimes Z), (s\otimes t) \sqcup  (s\otimes r)) \\
&=(X\otimes Y, s\otimes t)+ (X\otimes Z, s\otimes r).
\end{align*}
\item  
\begin{align*}
(X,s)\cdot (Y, t+f)&= (X\otimes Y, s\otimes (t+f))\\&=(X\otimes Y, (s\otimes t) + (s\otimes f))\\
&=(X\otimes Y, s\otimes t) + (X\otimes Y, s\otimes f) .
\end{align*}
\end{itemize}
\end{itemize}

Therefore, the product is well-defined.

\section{The functor  upper plus for $A$-fibered bisets } \label{sect functor upper }
Throughout this section, let $R$ denote a commutative ring and let the data $\mathcal{(G, S)}$
satisfying Axioms $(i)–(iii)$, $(vii)$ and $(viii)$.  Let $\mathcal{D := C(G, S)}$   be defined as
in Definition  \ref{S}  and let $\mathcal{S^+}$ and $\mathcal{D^+}$ be defined as
in Definition \ref{S^+}. The goal of this section is the construction of a functor $-^+ :\mathcal{F}^A_{D,R} \longrightarrow \mathcal{F}^A_{D^+,R}$ that generalizes the construction given  in [\cite{constructios+}, Section 5].\\

\begin{nota} 
If $G$ and $H$ are finite groups and   $U$ is  an $A$-fibered  $(G,H)$-biset, for   $u\in U$ and  $K\leq G$, we set 
\begin{align*}
K^u = \lbrace h \in H \mid \exists k \in K : (k, h) \in (G\times H)_u \rbrace.
\end{align*}
in the case $A = 1$, this notation has been introduced in \cite{unitp-grupos}.
It is easy to verify that 
\begin{align*}
& (K \times H)_u = \Delta(K) \ast (G\times H)_u  \ \ \ \  and   \ \ \ \  K^u = K \ast (G\times H)_u.&
\end{align*}
\end{nota}

\begin{rem} \label{condicionesDeSuma}
Let $G, \  H$  be  finite groups, $K$ be a subgroup of $G$, and let  $\alpha \in K^\ast$.  If   $U$ is an $A$-fibered  $(G,H)$-biset and  $g \in G $, one has: 
\begin{enumerate}[(a)]
\item If  $[A\backslash U/H]$ is a representative  set of $(A,H)$-orbits on  $U$. Then the set $\lbrace gu \mid u \in [A \backslash U /H] \rbrace$  is also a representative set  of  $(A,H)$-orbits of  $U$, for any $g\in G$. 
\item $K \leq p_1((G \times H)_u)$ if only if  $^gK \leq p_1((G\times H)_{gu})$.
 \item For  $u \in U$ and  $\lambda \in( K^u )^\ast$. The function $\phi_{u,2}$ is equal to  $\lambda $ in  $k_2((G\times H)_u) \cap K^u $ if only if the  function $\phi_{gu,2}$  is equal to $\lambda $ in $ k_2((G\times H)_{gu}) \cap (^gK)^{gu} $.
\item  For  $u \in U$ and  $\lambda \in( K^u )^\ast$. One has  $ \phi_u \ast \lambda = \alpha$ if only if  $\phi_{gu} \ast \lambda ={ ^g \alpha}$.

\end{enumerate}
\end{rem}

\begin{defi}\label{F^+}
For  $G\in \mathcal{G}$, we define  $F^+(G)$ as  the set of all elements  $(x_{(K, \lambda)})_{(K,\lambda ) \in \mathcal{M}^ \prime(G)} \in \oplus_{(K, \lambda )\in \mathcal{M}^\prime(G)} F(K) $  such that  $$^g x_{(K,\lambda)}:=F(Iso ^{^gK}_K) (x_{(K, \lambda)})=x_{(^gK, ^g\lambda)}$$  for  all  $(K,\lambda )\in \mathcal{M}^\prime (G)$. It  is easy to verify that  
\begin{align*}
F^+(G) \leq  \bigoplus_{(K, \lambda)\in \mathcal{M}^\prime(G)} F(K).
\end{align*}
For any   $G$,  $H \in  \mathcal{G}$, and any  $A$-fibered $(G,H)$-biset $U$ such that all pairs stabilizers   $((G\times H)_u, \phi_u) \in \mathcal{S}^+(G,H)$.  We define 
\begin{equation} \label{F^+ (U)}
\begin{split}
F^+([U])=F^+(U): F^+(H)&\longrightarrow F^+(G) \\
(x_{L,\lambda})_{(L, \lambda) \in \mathcal{M}^\prime (H)} &\longmapsto (y_{(K, \alpha)})_{(K, \alpha) \in \mathcal{M}^\prime (G)}
\end{split}
\end{equation}
where
\begin{align*}
y_{(K, \alpha)}= \sum_{\begin{subarray}{l} u \in [A\backslash U /H]\\
K \leq p_1(G\times H)_u) \\ \phi_{u,2} \vert _{W_{u,K}}= \lambda \vert _{W_{u,K}} \\ \phi_u \ast \lambda =\alpha \end{subarray}}  F\left( \left[ \dfrac{K \times K^u}{(K\times H)_u, \phi_u } \right] \right)  (x_{(K^u, \lambda)}),
\end{align*}
and the set $[A\backslash U /H]$ denotes a set of representatives of the   $(A,H)$-orbits  of $U$, and $W_{u, K} = k_2((G\times H)_u) \cap K^u)$. Now we show that the  above expression is well-definite: 

\begin{itemize}
\item For any  $K\in \mathcal{G}$, by Axiom $(viii)$  of  $(\mathcal{G, S}) $,   one has  $$\Delta(K)\ast (G\times H)_u=(K\times H)_u \in \mathcal{G}$$  and  $(\Delta(K),1) \ast ((G\times H)_u , \phi_u)= ((K\times H)_u, \phi_u) \in \mathcal{S} (K, K^u)$. Then  $F$ can be applied to the class   $\left[ \dfrac{K \times K^u}{(K\times H)_u, \phi_u } \right]$.
\item The expression in (\ref{F^+ (U)}) does not depend on the choice of representative   $[A\backslash U/H]$.  For any  $a\in A$, $h\in H$ and  $u\in U$, one has:
\begin{enumerate}
\item $K^{auh}=K \ast (G\times H)_{auh}=K\ast (G\times H)_{u} ^{ (1,h^{-1})}={^h (K^u)}$
\item \begin{align*}
((K\times H)_{uh}, \phi_{uh})&= (\Delta(K),1) \ast ((G\times H)_u, \phi_u)^{(1,h^{-1})}   \\
&= ((K\times H)_u , \phi_u)^{(1, h^{-1})}
\end{align*}
\end{enumerate}
Thus 
\begin{align*}
F\left( \left[  \dfrac{K\times K^{auh}}{(K\times H)_{auh}, \phi_{auh}} \right]\right)  (x_{K^{auh}})=F \left( \left[  \dfrac{K\times K^{u}}{(K\times H)_{u}, \phi_{u}} \right] \right) ( x_{K^{u}}).
\end{align*}
It follows that does not depend  on  the choice of  $(A,H)$-orbits of $U$.
\item Does not depend on the choice of representative  of $[U]$. 
Let  $V$ be an $A$-fibered   $(G,H)$- biset isomorphic to  $U$, then there exists  \- $\psi : U\longrightarrow V$ an isomorphism  in $_Gset_H^A$, thus: 
\begin{itemize}
\item $((G\times H)_u, \phi_u)=((G\times H)_{\psi(u)}, \phi_{\psi(u)})$.
\item $K^u= K^{\psi(u)}$. 
\item $((K\times H)_u, \phi_u)=((K\times H)_{\psi(u)}, \phi_{\psi(u)})$.
\item The set $\lbrace \psi(u) \in V \mid u \in [A\backslash U/ H]\rbrace$  is a set of representative  of   $(A,H)$-orbits of $V$.
\end{itemize} 
It follows that 
\begin{align*}
F^+(U)\left( (x_{( L,\alpha)})_{(L, \alpha) \in \mathcal{M}^\prime (H)} \right) = F^+(V)\left( (x_{(L,\alpha)})_{(L,\alpha) \in \mathcal{M}^\prime ( H)}\right) .
\end{align*}
\item Now, we show that  $(y_{(L,\alpha)})_{(L,\alpha)\in \mathcal{M}^\prime(G)} =F^+(U)( (x_{(K,\lambda)})_{(K,\lambda) \in \mathcal{M}^\prime (K)})$ is an element of $F^+(G)$. One has $^g y_{(L,\alpha)}$ is equal to 
\begin{align*}
\sum_{\begin{subarray}{l} u \in [A\backslash U /H]\\
K \leq p_1(G\times H)_u) \\ \phi_{u,2} \vert _{W_{u,K}}= \lambda \vert _{W_{u,K}} \\ \phi_u \ast \lambda =\alpha \end{subarray}}   F\left(\left[ \dfrac{^g K \times K}{^{(g,1)}\Delta (K), 1}\right] \ast \left[\dfrac{ K \times K^u }{(K\times H)_u, \phi_u} \right]  \right) (x_{(K^u, \lambda)}), 
\end{align*}
where $W_{u,K}= k_2((G\times H)_u) \cap K^u$.  Note that   $K^u=(^g K)^{gu}$,  $ (^{(g,1)}\Delta(K), 1)\ast ((K\times H)_u, \phi_u)=((^gK \times H)_{gu},  \phi_{gu})$,  by Remark \ref{condicionesDeSuma}, one has  $^g y_{(K,\alpha)}$  is equal to  

\begin{align*}
\sum_{\begin{subarray}{l} gu \in [A\backslash U /H]\\
^gK \leq p_1(G\times H)_{gu}) \\ \phi_{gu,2} \vert _{W_{gu, ^gK}}= \lambda \vert _{W_{gu,^gK}} \\ \phi_{gu} \ast \lambda ={^g\alpha} \end{subarray}}   F\left( \left[\dfrac{ ^g K \times (^gK)^{gu} }{(^gK\times H)_{gu}, \phi_{gu}} \right]  \right) (x_{(^gK)^{gu}, \lambda}), 
\end{align*}
this is equal to $y_{(^gK, ^g\alpha)}$.\\
Let's note that if $\mathcal{D}$ satisfies condition $k_2$, see section \ref{axioms}, the condition $ \phi_{u,2} \vert {W{u,K}} $ being equal to $ \lambda \vert {W{u,K}}$ in the definition of $y_{(K,\alpha)}$ is trivial, since $W_{u,K}=1_H$.
Clearly, $F^+([U])$ is also an $R$-module homomorphism. Since the sum in (\ref{F^+ (U)}) is additive in $[U]$, we also obtain an induced group
homomorphism
\begin{align*}
F^+ : B^A(G, H) \longrightarrow Hom_R(F^+(H), F^+(G)).
\end{align*}

\end{itemize} 
\end{defi}
\begin{lem} \label{L^u}
Let $U$ be an $A$-fibered  $(G,H)$-biset, let  $V$ be an  $A$-fibered  $(H,K)$-biset. For any   $u\otimes v \in U\otimes_{AH} V$ and  $L \leq G$, one has 
\begin{enumerate}[(a)]
\item  $(L^u)^v =L^{u\otimes v}$.
\item Assume that  $U$ is right-free. Then $L\leq p_1( (G\times K)_{u\otimes v})$  if only if $L\leq p_1((G\times H)_u)$ y $L^u \leq p_1((H\times K)_v)$.
\end{enumerate}
\end{lem}

\begin{thm} \label{F^+ conjuga}

Let $\mathcal{(G, S)}$ be as \ref{axioms}, such that satisfying Axioms  $(i)$, $(ii)$, $(iii)$, $(vii)$ and  $(viii)$ and assume that for every  $G, H\in \mathcal{G}$ and for every  $ (D,\phi) \in \mathcal{S}(G, H)$ one has    $k_2(D) = \left\lbrace 1\right\rbrace $.  Set $\mathcal{D:= C}^A\mathcal{(G, S)}$ and $\mathcal{D^+ = C(G, S^+)}$. Then  the constructions in Definition \ref{F^+}  define an $R$-lineal functor  $\textbf{--}^+ : \mathcal{F_{D,R}} \longrightarrow \mathcal{F_{D^+,R}}.$
\end{thm}
\begin{proof}
Let $G, H$, $K \in \mathcal{G}$, and let  $U$ be an $A$-fibered $(G.H)$-biset such that the  pairs s for any element  is element of  $\mathcal{S}^+ (G,H)$, and let  $V$ be an $A$-fibered  $(H,K)$-biset, such that the  pair stabilizer for any element  is elements  of  $\mathcal{S}^+ (H,K)$, we show that
\begin{align*}
F^+(U\otimes_{AH} V)= F^+(U)\circ F^+(V) = F^+(K)\longrightarrow F^+(G).
\end{align*}
Let $c=(c_{(N,\beta)})_{(N,\beta) \in \mathcal{M}^\prime(K)} \in F^+ (K)$. Then for any $(L, \lambda) \in \mathcal{M}^\prime (G)$, the  $(L, \lambda)$-component  of  \- $F^+(U\otimes_{AH} V) (c)$ is equal to  
\begin{align*}
\sum_{\substack{u\otimes v \in [A\backslash (U\otimes_{AH} V)/K]\\L \leq p_1((G\times K)_{u\otimes v}) \\ \phi_u \ast \beta = \lambda}} F \left( \left[ \dfrac{L\otimes L^{u\otimes v}}{(L\times K)_{u\otimes v}, \phi_{u\otimes v}}\right]\right)   \left( c_{(L^{u\otimes v}, \beta)}\right). 
\end{align*}
Next, we compute the right-hand side.  We set  $b:=(b_{(M, \alpha)})_{(M, \alpha) \in \mathcal{M}^\prime(H)}= F^+ (V)(c)$,  then
\begin{align*}
b_{(M,\alpha)} =  \sum_{\substack{v\in [A\backslash V/K]\\M\leq p_1((H\times K)_v)\\ \phi_v \ast \psi = \alpha}}  F^+ \left( \left[ \dfrac{M\times M^v}{(M\times K)_v, \phi_v}\right]  \right) \left( c_{(M^v, \psi)}\right). 
\end{align*}
 and set  $a=(a_{(L, \lambda)})_{(L, \lambda )\in \mathcal{M}^\prime (G)}= F^+(U)(b)$, then 

\begin{equation}\label{a_l}
\begin{split}
a_{(L, \lambda)}&= \sum_{\substack{u\in [A\backslash U/H]\\L \leq p_1((G\times H)_{u})\\ \phi_u \ast \alpha = \lambda}} F \left( \left[ \dfrac{L\times L^u}{(L\times H)_u, \phi_u} \right] \right) \left( b_{(L^u, \alpha) } \right)
\end{split}
\end{equation} 
this is equal to 
\begin{align*}
\sum_{\substack{u\in [A\backslash U/H]\\L \leq p_1((G\times H)_{u}) \\ \phi_v \ast \psi = \alpha \\ v\in [A\backslash V/K]\\L^u \leq p_1((H\times K)_{v}) \\ \phi_u \ast \alpha = \lambda }}     F \left( \left[ \dfrac{L\times L^u}{(L\times H)_u, \phi_u} \right]  \otimes_{AL^u} \left[ \dfrac{L^u \times (L^u)^v}{(L^u\times K)_v , \phi_v}\right] \right) \left( c_{((L^u)^v, \psi)} \right).
\end{align*}
Note that:
\begin{itemize}
    \item $p_2((L\times H)_u)=L^u = p_1((L^u\times K)_v)$ 
    \item $\phi_u =\phi_v$ in  \- $k_2((L\times H)_u)\cap k_1((L^u\times H)_v)= \lbrace 1 \rbrace$,
\end{itemize}
 by Mackey's  formula 
\begin{align*}
 \left[ \dfrac{L\times L^u}{(L\times H)_u, \phi_u} \right]  \otimes_{AL^u} \left[ \dfrac{L^u \times (L^u)^v}{(L^u\times K)_v , \phi_v}\right] = \left[ \frac{L\times L^{u\otimes v}}{((G\times H)_u, \phi_u) \ast (L^u\times K)_v,\phi_v)}\right]. 
\end{align*}
Now, by Lemma \ref{L^u},  one has  (\ref{a_l})  is equal to 
\begin{align*}
\sum_{\substack{v\in [A\backslash V/K]\\ u \in [A \backslash U/H]\\ L\leq p_1((G\times K)_{u\otimes v}) \\ \phi_v \ast \psi =\alpha\\ \phi_u  \ast \alpha = \lambda}} F\left(\left[ \frac{L\times L^{u\otimes v}}{((G\times H)_u, \phi_u) \ast (L^u\times K)_v,\phi_v)}\right]  \right) \left(  c_{(L^{u\otimes v}, \psi)}\right).
\end{align*}
Since $U$ is right-free, one  has for any sets of representatives   $[A\backslash U /H]$ and  $[A\backslash V/K]$,  the set  \- $ \lbrace u\otimes v \mid u \in [A\backslash U /H] \text{ and } v\in [A\backslash  V/K]\rbrace$  is a set of representatives of $A\backslash (U\otimes_{AH} V)/ K$. Moreover, by Lemma \ref{estProd}, one has   

\begin{align*}
a_{(L, \lambda)}= \sum_{\substack{u\otimes v\in [A\backslash (U\otimes_{AH} V)/K]\\\\ L\leq p_1((G\times K)_{u\otimes v}) \\ \phi_v \ast \psi =\alpha\\ \phi_u  \ast \alpha = \lambda}} F\left(\left[ \frac{L\times L^{u\otimes v}}{(L \times K)_{u\otimes v}, \phi_{u\otimes v}}\right]  \right) \left(  c_{(L^{u\otimes v}, \phi_{u\otimes v})}\right),
\end{align*}
Since $\alpha$ does not appear in the sum,  we can replace the conditions   $ \phi_v \ast \psi =\alpha$ and $ \phi_u  \ast \alpha = \lambda$ by  $ \phi_u \ast \phi_v \ast \psi = \lambda$, and by lemma \ref{estProd}, one has  $\phi_u \ast \phi_v = \phi_{u\otimes v}$. Thus, 
\begin{align*}
F^+(U\otimes_{AH} V)(c)= F^+(U)\circ F^+(V)(c).
\end{align*}
That $F^+$ preserves identity morphisms and is $R$-linear follows immediately
from the definitions.

\end{proof}

\subsection{Multiplicative Structure of $F^+(G)$}
\begin{thm}
Let the pair $\mathcal{(G,S)}$ satisfy axioms $(i)$ to $(iii)$, define $ \mathcal{D:=C}^A \mathcal{(G,S)}$, $ \mathcal{D^+:=C}^A \mathcal{(G,S^+)}$, and let $F\in \mathcal{F}_{\mathcal{D},R}^{ A}$ such that $F(G)$ has a multiplicative structure for every $G\in \mathcal{G}$. If $(\mathcal{G,S})$ also satisfies axioms $(vii)$ and $(viii)$, and condition $k_2$, then $F^+(G)$ has a multiplicative structure for every $G\in \mathcal{G}$.
\end{thm}

\begin{proof}
For every $G \in \mathcal{G}$, we will give $F^+(G)$ a multiplicative structure, which will be defined entrywise.

\end{proof}

\section{ The mark morphism} \label{sect  mark morphism}
The goal of this section is given a morphism between the construction lower plus and the construction upper plus for an $A$-fibered biset functor, such that this morphism  behaves similarly to an isomorphism of modules in each evaluation. 
Let, $\mathcal{(G,S)}$ as  Definition \ref{axioms},  such that satisfying  Axioms $(i)$-$(iv)$, $(vi)$, $(vii)$ and  $(viii)$  and the condition $k_2$. Set $\mathcal{D:=(G,S)}$, $\mathcal{D_+:=(G,S_+)}$ and  $\mathcal{D^+:=(G,S^+)}$, 
by proposition \ref{D^+} (c) and  \ref{D_+} (c), one has $\mathcal{E:=D^+ =D_+}$.  For $F\in \mathcal{F}^A_{\mathcal{E},R}$ and $G\in \mathcal{G}$, we define the  $R$-lineal map  $$m_{F,G}: F_+(G) \longrightarrow F^+(G)$$
which send the class $[X,s]\in F_+(G)$  to $(a_{(L \lambda)})_{(L,\lambda)\in \mathcal{M}^\prime(G)} \in F^+(G)$ with 
\begin{align*}
a_{(L, \lambda)} =\sum_{\begin{subarray}{l} x\in[ X/A]\\(L, \lambda) \leq (G_x, \phi_x) \end{subarray}} F\left( Res^{G_x}_L  \right) (s(x)), 
\end{align*}
where $[X/A]$  is a set of representative $A$-orbits of  $X$, this map we call "The mark morphism".  The mark morphism is well-defined, that does not depend on the choice of representative  of $[X,s]$.  

\begin{thm}
Assume that  $\mathcal{(G,S)}$ satisfies Axioms $(i)$-$(iv)$, $(vi)$, $(vii)$ and  $(viii)$ in \ref{axioms}. Additionally, the condition $k_2$.  Let  $G,H\in \mathcal{G}$ and let $U$ be an $A$-fibered   $(G,H)$-biset such that all stabilizer   pairs are elements of $\mathcal{S}^+(G,H)=\mathcal{S}_+(G,H)$. Then,  for any  functors  $F \in \mathcal{F}^A_{\mathcal{E},R}$ ,  the diagram \\
\begin{equation*}
\xymatrix{
F_+(H)\ar[r]^{m_{F,H}} \ar[d]_{F_+([U])} & F^+(H) \ar[d]^{F^+([U])}\\
F_+(G)\ar[r]_{m_{F,G}} & F^+(G) 
}
\end{equation*}
commutes. In particular, the $R$-linear maps $m_{F,G}$, $G \in \mathcal{G}$, define a morphism $m_F : F_+ \longrightarrow F^+$ in $\mathcal{F}^A_{\mathcal{E},R}$ which we will call the mark morphism associated with F. 
\end{thm}
\begin{proof}
 Let $[X,s]\in F_+(G)$ and $(L,\alpha) \in  \mathcal{M}^\prime(G)$.  By definition, the $(L,\alpha)$-coordinate  of  $m_{F,G}(F_+(U)([X,s]))$ is equal to   
\begin{align*}
&\sum_{\substack{ u\otimes x \in [U\otimes_{AH} X/A] \\ (L, \alpha)\leq (G_{u\otimes x}, \phi_{u\otimes x})}} F(Res^{G_{u\otimes x}}_L)(U(s)(u\otimes x))\\
= &\sum_{\substack{ u\otimes x \in [U\otimes_{AH} X] \\ (L,\alpha)\leq (G_{u\otimes x},\phi_{u\otimes x})}} F \left( \left[  \frac{L\times G_{u\otimes x}}{\Delta(L),1}\right]\otimes_{AG_{u\otimes x}} \left[\frac{G_{u\otimes x }\times H_x}{(G\times H_x)_u, \phi_{u,x}} \right]  \right) (s(x))\\
=&\sum_{\substack{ u\otimes x \in [U\otimes_{AH} X] \\ (L,\alpha)\leq (G_{u\otimes x},\phi_{u\otimes x})}} F \left( \left[ \frac{L \times H_x}{(L\times H_x)_u, \phi_{u,x}} \right]  \right) (s(x)).
\end{align*}
 
On the other hand,  we set  $(a_{(K,\lambda)})_{(K,\lambda)\in \mathcal{M}^\prime(G)}= m_{F,G} ([X,s]) $  and  $$b:=(b_{(L, \alpha)})_{(L, \alpha) \in \mathcal{M}(G)}:=F^+([U])(m_{F,G} ([X,s]) ),$$ then the $(L,\alpha)$-coordinate of $B$ is equal to 
\begin{align*}
b_{(L,\alpha )}&=\sum_{\substack{ u \in [A \backslash U/H]\\ L\leq p_1((G\times H)_{u})\\ \phi_u \ast \lambda= \alpha}} F\left(\left[ \frac{L\times L^u}{(L\times H)_u , \phi_u}\right]  \right) \left( a_{(L^u,\lambda)}\right) \\
&=\sum_{\substack{ u \in [A \backslash U/H]\\ L\leq p_1((G\times H)_{u})\\ \phi_u \ast \lambda= \alpha}} F\left(\left[ \frac{L\times L^u}{(L\times H)_u , \phi_u}\right]  \right) \left( \sum_{\substack{ x \in [X/A] \\ (L^u,\lambda) \leq (H_x,\phi_x)}} F(Res^{H_x}_{L^u}) (s(x)) \right) \\
&=  \sum_{\substack{ u \in [A \backslash U/H]\\  x \in [X/A] \\ L\leq p_1((G\times H)_u )\\ \phi_u \ast \lambda=\alpha\\ (L^u, \lambda) \leq (H_x, \phi_x)}} F\left(\left[  \frac{L\times L^u}{(L\times H)_u, \phi_u)} \right]  \otimes_{AL^u} \left[ \frac{L^u \times H_x}{\Delta(L^u), 1} \right] \right)  (s(x)), 
\end{align*}
by  Lemma \ref{L^u}, this is equal to 
\begin{align*}
&=  \sum_{\substack{ u \in [A \backslash U/H]\\  x \in [X/A] \\ \phi_u \ast \lambda=\alpha\\ (L^u, \lambda) \leq (H_x, \phi_x) \\ L\leq G_{u\otimes x}}} F\left(\left[  \frac{L\times L^u}{(L\times H)_u, \phi_u)} \right]  \otimes_{AL^u} \left[ \frac{L^u \times H_x}{\Delta(L^u), 1} \right] \right)  (s(x)). 
\end{align*}
Since $U$ is a right-free,  the set $$\lbrace u\otimes x \in U\otimes_{AH
}  X \mid x \in [X/A]  \text{ and } u\in [A\backslash U/H] \rbrace$$ is a set of representative of  $A$-orbits  of $U\otimes_{AH} X$, and we can  replace  the conditions $\phi_u \ast \lambda=\alpha$ and  $ (L^u, \lambda) \leq (H_x, \phi_x)$ by     $$(L,\phi_u\ast \lambda)=(L,\alpha) \leq (G_{u\otimes x}, \phi_u\ast \phi_x)=(G_{u\otimes x }, \phi_{u\otimes x}). $$ Note that  $p_1((G\times H_x)_u)=G_{u\otimes x}$ and
\begin{align*}
(\Delta(L),1)\ast ((G\times H_x)_u, \phi_{u,x} )&=((L\times H_x)_u, \phi_{u,x})\\&=((L\times H)_u, \phi_u) \ast (\Delta(L^u), 1).
\end{align*}
Thus the $(L, \alpha) $-coordinate  of     $F^+([U])\left( m_{F,G}([X,s])\right) $
is equal to the $(L, \alpha)$-coordinate of  \- $m_{F,G}(F^+([U])([X,s]))$.
\end{proof}
The following definition gives a map that is close to an inverse to the mark morphism.
\begin{defi}\label{n_F,G}
Assume that $\mathcal{(G,S)}$ satisfies  Axioms  $(i)$-$(iv) $ $(vi)$  $(vii)$ and $(viii)$ in  \ref{axioms} and set $\mathcal{D:=C(G,S)}$. Let  $F\in \mathcal{F}^A_{\mathcal{D}, R}$ and $G\in \mathcal{G}$. We define the map  $n_{F,G} : F^+(G)\longrightarrow F_+(G)$ by $n_{F,G} \left(  (a_{K, \lambda,})_{(K,\lambda) \in \mathcal{M}^\prime (G)}\right)$ is equal to, 
\begin{align*}
 \sum_{\substack{ (L,\psi), \ (K, \lambda) \in \mathcal{M}^\prime (G)\\ (L, \psi)\leq (K, \lambda)}} |L|\mu ((L, \psi), (K, \lambda))[L, \psi, Res^K_L a_{(K,\lambda)}].
\end{align*}
Here, $\mu$ denote the  Möbius function of the poset  $\mathcal{M}^\prime (G)$.
\end{defi}
\begin{prop}
Let  $\mathcal{(G, S)}$, $\mathcal{D}$, $F$ and  $G$ be as in Definition  \ref{n_F,G}.  Then $n_{F,G} \circ m_{F,G}= |G| \cdot Id_{F_+(G)}$ and $m_{F,G} \circ n_{F,G} =|G| \cdot Id_{F^+(G)}$.
\end{prop}
\begin{proof}
Let  $[X, s] \in F_+(G)$, one has  $n_{F,G}(m_{F,G}([X,s]))$ is equal to 
\begin{align*}
& \sum_{\substack{ (L,\psi), \ (K, \lambda) \in \mathcal{M}^\prime (G)\\ (L, \psi)\leq (K, \lambda)}} |L|\mu ((L, \psi), (K, \lambda))[L, \psi, Res^K_L (m_{F,G}([X,S]))_{(K,\lambda)}]\\
 &= \sum_{\substack{ (L,\psi), \ (K, \lambda) \in \mathcal{M}^\prime (G)\\ (L, \psi)\leq (K, \lambda) \\ x\in [X/A]\\\ (K,\lambda)\leq (G_x, \phi_x)}} |L|\mu ((L, \psi), (K, \lambda))\left[ L, \psi, Res^K_L \left(  F(Res^{G_x}_K) (s(x))\right) \right] \\
\end{align*} 
this is equal to 
\begin{align*} 
 \sum_{\substack{ (L,\psi), \ (K, \lambda) \in \mathcal{M}^\prime (G)\\ (L, \psi)\leq (K, \lambda)}} |L|      \sum_{\substack{ x\in [X/A] \\ (L,\psi)\leq (K,\lambda)\leq (G_x,\phi_x )} }  \mu ((L, \psi), (K, \lambda))[L, \psi, F(Res^{G_x}_L) (s(x))]
\end{align*}
where $(m_{F,G}([X,S]))_{(K,\lambda)}$ denotes the  $(K,\lambda)$-component of $m_{F,G}([X,S])$. By the standard  properties of $\mu$, the last equation is equal to, 
\begin{align*}
n_{F,G}(m_{F,G}([X,s]))&=\sum_{\substack{ (L,\psi),\in \mathcal{M}^\prime(G)}}      \sum_{ \substack{x\in [X/A]\\ (L,\psi)=(G_x, \phi_x)}}  | G_x| [G_x, \phi_x, F(Res^{G_x}_{G_x}) (s(x))]\\
&= |G| \sum_{x\in [G\backslash X/ A]}  [G_x, \phi_x, s(x)]\\
&= |G| [X,s]
\end{align*}
since $$|\lbrace  y \in [X/A] \mid [G_y, \phi_y, s(y)] =[G_x, \phi_x, s(x)] \rbrace|= |G|/|G_x|.$$
The second equality is a similar adaptation of the proof of  Proposition 2.4. of \cite{canonicalinductionformulae}.  For any $(a_{(K,\lambda)})_{(K,\lambda)\in \mathcal{M}^\prime (G)} \in F^+(G)$, one has  $n_{F,G}((a_{(K,\lambda)})_{(K,\lambda)\in \mathcal{M}^\prime (G)})$ is equal to 
\begin{align*}
\sum_{\substack{ (L, \psi), \ (K,\lambda) \in \mathcal{M}^\prime (G) \\ (L,\psi) \leq (K, \lambda)}} |L| \mu ((L, \psi), (K, \lambda)) [L, \psi, F(Res^K_L) (a_{(K,\lambda)})]
\end{align*}
It follows that,  $m_{F,G} \left( n_{F,G}((a_{(K,\lambda)})_{(K,\lambda)\in \mathcal{M}^\prime (G)}) \right) $ is equal to 
\begin{align*}
\sum_{\substack{ (L, \psi), \ (K,\lambda) \in \mathcal{M}^\prime (G) \\ (L,\psi) \leq (K, \lambda)}} |L| \mu ((L, \psi), (K, \lambda)) m_{F,G}([L, \psi, F(Res^K_L) (a_{(K,\lambda)})]).
\end{align*}
Note that  $m_{F, G} ([L, \psi, F(Res^K_L)(a_{(K, \lambda)})] = (b^{(L,\psi)}_{(T, \alpha)})_ {(T, \alpha) \in \mathcal{M}^\prime(G)}$, where
\begin{align*}
b_{(T,\alpha)}^{(L,\psi)} = \sum_{ \substack{gL \in [G/L] \\ (T,\alpha ) \neq ^g(L, \psi )}} F(Res^{^gL}_T  \circ Res^ {^g K}_{^gL}) (^ga_{K,\lambda}).
\end{align*}
Then, the  $(T,\alpha)$- component of  $m_{F,G} \left( n_{F,G}((a_{(K,\lambda)})_{(K,\lambda)\in \mathcal{M}^\prime (G)}) \right) $  is equal to 
\begin{align*}
\sum_{g\in G} \sum _{(T^g, \alpha^g) \leq (L,\psi)} \frac{|L|}{|L|} \sum_{(L,\psi)\leq (K,\lambda)}  \mu ((L,\psi), (K,\lambda ))F(Res^{^gK}_T)(^g(a_{K,\lambda}))
\end{align*}
by Möbius inversion [Proposition 2 of the section  3 of  \cite{MobiusGCRota} this sum collapses to   $^g a_{(T^g, \alpha^g)}=a_{(T,\alpha)}$. 
Thus  $(T,\alpha)$-component of   $m_{F,G} \left( n_{F,G}((a_{(K,\lambda)})_{(K,\lambda)\in \mathcal{M}^\prime (G)}) \right)$ is $|G| a_{(T,\alpha)}$.
\end{proof}

\begin{cor}
Let $\mathcal{(G,S)}$, $\mathcal{D}$, $F$ and $G$ be as in Definition \ref{n_F,G}. If $|G|$ is invertible in $R$, then  $m_{F,G}$ and  $|G|^{-1} n_{F,G}$ are mutually inverse isomorphism. 
\end{cor}

\begin{cor}
Let $\mathcal{(G,S)}$, $\mathcal{D}$, $F$ and  $G$ be as in Definition  \ref{n_F,G}. If  $F_+(G)$ has trivial $|G|$-torsion, then $m_{F,G}$ is injective.
\end{cor}

\section{ Green Functors  } \label{sect Green funct}
This section is devoted to studying   of the functors lower and upper plus. Throughout this section, the class $\mathcal{G}$ is closed under the direct product; we will denote the category $\mathcal{D:=C} ^A\mathcal{(G,S})$ and consider the subcategory of biconjunctive functors, $A$-fibered over $\mathcal{D}$, denoted  by $\mathcal{F}^{A}_{\mathcal{D},R}$,  which  consists of functors of Green biconjunctive $A$-fibered (Definition \ref{Green-fiberd}).  For  $\psi :G \longrightarrow H$ an  isomorphism of groups, we set  
\begin{align*}
Iso(\psi)^A :=\left[ \dfrac{H\times G}{^{(\psi,1 )} \Delta(G), 1}\right].
\end{align*}
\begin{defi}\label{Green-fiberd}
Let $F \in \mathcal{F}_{\mathcal{D},R}^{A}$,  $F$  is a Green $A$-fibered biset functor if it is an $A$-fibered  biset functor quipped with bilinear products
\begin{align*}
\times : F(G)& \times F(H) \longrightarrow F(G\times H)\\
(c,d) &\longmapsto (c\times d)
\end{align*}
For any  $G$, $H \in \mathcal{G}$, and an element   $\mathcal{E}_F \in F(1)$, satisfying the following conditions:
\begin{itemize}
\item  Associativity. Let $G$, $H$ and  $K$ finite groups in  $\mathcal{G}$, and let  $\alpha:G \times(H\times K) \longrightarrow (G\times H) \times K$  the canonical isomorphism, then for any $x\in F(G)$, $y\in F(H)$ and $z\in F(K)$
\begin{align*}
(x\times y) \times z= F(Iso(\alpha)^A)(x\times (y \times z)).
\end{align*}
\item Identity element. Let $G$ be a finite group of $ \mathcal{G}$. Let  $\rho_G : 1\times G \longrightarrow G$ and $\lambda_G: G \times 1 \longrightarrow G$ denote the canonical isomorphisms. Then for any   $x \in F(G)$
\begin{align*}
x=F(Iso(\rho_G)^A)(\mathcal{E}_F \times x)= F(Iso(\lambda_G)) (x\times \mathcal{E}_F).
\end{align*}
\item Functoriality: If  $\phi : G \longrightarrow H $ and  $\psi: G^\prime \longrightarrow H^\prime$ morphisms in  $R\mathcal{D}$, then for any  $x\in F(G)$ and  $y\in F(H)$
\begin{align*}
F(\phi \otimes_A \psi) (x\times y)= F(\phi)(x)\times F(\psi)(y).
\end{align*}
\end{itemize}
Let $F_1, F_2 \in \mathcal{F}_{\mathcal{D},R}^{A}$ be Green $A$-fibered biset functors on $\mathcal{D}$ over $R$. A morphism of Green $A$-fibered biset functors between $F_1$ and $F_2$ is a morphism $\eta : F_1 \longrightarrow F_2$ in the category $\mathcal{F}_{D,R}^A $ with
the additional property that, 
\begin{equation*}
\xymatrix{
F_1(G)\times F_1(H)\ar[r]_{   \eta_G\times \eta_H } \ar[d]_{\times } &  F_2(G)\times F_2(H) \ar[d]^{\times }\\
F_1(G\times H)\ar[r]_{\eta_{G\times H}} & F_2(G\times H) 
}
\end{equation*}
For all $G$, $H$ in $\mathcal{G}$, and $\eta_1(\mathcal{E}_{F_1})=\mathcal{E}_{F_2}$.
\end{defi}
We denote the category of Green $A$-fibered biset functors on 
 $\mathcal{D}$ over $R$ by $\mathcal{F}_{D,R}^{A,\mu}$.
\begin{rem}
Let $G$ and $H$ finite  groups. For any $A$-fibered  $G$-set   $X $ and for any $A$-fibered $H$-set $Y$, one has $(G\times H)_{x\otimes y} = G_x \times H_y$,  for all $x\otimes y \in X\otimes_A Y$. 
\end{rem}

The followings lemmas  shows that the constructions  $F_+$ and  $F^+$ preserve the property of being a Green fibered biset functors.
\begin{lem}\label{green_+}
Let  $\mathcal{(G,S)}$ satisfy Axioms $(i)$- $(iii)$, $(iv)$ and $(vi)$ in  \ref{axioms}, and set  $\mathcal{D}= \mathcal{C}^A \mathcal{(G,S)}$,  $\mathcal{D_+=C}^A \mathcal{(G,S_+)}$. For any  $F\in \mathcal{F}_{\mathcal{D},R}^{A,\mu}$,  then $F_+ \in \mathcal{F}_{\mathcal{D}_+, R}^{A,\mu}$. 
\end{lem}
\begin{proof}
We define the bilinear product of  $F_+$.  Let  $G$, $H$ be elements of  $\mathcal{G}$,  we define
\begin{align*}
\times: F_+ (G) \times F_+(H) &\longrightarrow  F_+(G\times H)\\
([X,s], [Y, t]) &\longmapsto  [X\otimes_A Y, s\times t],
\end{align*}
where  
\begin{align*}
s \times t : X\otimes_A Y &\longrightarrow \coprod_{x\otimes y \in [X\otimes_A Y/A]} F((G\times H)_{x\otimes y})\\
x\otimes y &\longmapsto s(x)\times t(y).
\end{align*} 
Since  $s$ and  $t$ are section  $A$-equivariant, one has  $s \times t$ is well-define, 
and set  $\mathcal{E}_{F_+}=[1, s_{\mathcal{E}_F}] \in F_+(1)$. Now we prove that, this bilinear product and $\mathcal{E}_{F_+}$ satisfies  the properties stated in the Definition \ref{Green-fiberd}: \\
\textbf{Associativity}: Let $G$, $H$, $K$ be elements of  $\mathcal{G}$  and  let  $$\alpha : G \times(H\times K) \longrightarrow (G\times H) \times K$$  be the canonical isomorphism. For any $[X,s]\in F_+(G)$, $[Y,t] \in F_+(H)$ and  $[Z,r]\in F_+(K)$, one ha\begin{align*}
F_+(Iso (\alpha)^A)\left( [X,s]_G\times ([Y,t]_H\times [Z,r]_K )\right)
\end{align*}
is equal  to 
\begin{align*}
[Iso (\alpha)^A \otimes_{G\times (H \times K)} X\otimes_A (Y\otimes_A Z), Iso(\alpha)^A (s\times (t\times r))]_{(G \times H)\times K}.
\end{align*}\label{iso_nat_GKH}
Note that there exists an isomorphism of $A$-fibered bisets between
\begin{equation}
Iso (\alpha)^A \otimes_{(G\times H) \times K} X\otimes_A (Y \otimes_A Z ) \cong (X \otimes_A Y) \otimes_A Z.
\end{equation}
since $Iso(\alpha)^A$ is $G\times (H\times K)$-transitive, one has  every element of  $(G\times H)\times K$-set \-$Iso (\alpha)^A \otimes_{(G\times H) \times K} (X\otimes_A Y) \otimes_A Z $  we can see as $1\otimes_{A((G\times H)\times K)} (x\otimes (y\otimes z))$, for some  $x\in X$, $y \in Y$ and  $z\in Z$. then 
\begin{align*}
Iso(\alpha)^A (s\times (t\times r))(1\otimes(x\otimes (y \otimes z))) 
\end{align*}
is equal to  
\begin{align*}
 F\left(\left[ \dfrac{((G\times H)\times K)_{x\otimes (y\otimes z)} \times (G\times (H\times K))_{(x\otimes y )\otimes z}}{^{(\alpha,1)} \Delta(G\times (H\times K))_{(x\otimes y )\otimes z}), 1  }\right]  \right) (s(x)\times (t(y)\times r(z))),
\end{align*}
by associativity of bilinear product  $\times$ of  $F$, this is equal to 
$$(s \times(t \times r))(x\otimes (y\otimes z)).$$
\textbf{Identity element}: Let  $[X,s] \in F_+(G)$, one has a natural isomorphism of $A$-fibered $G$-sets  between  
\begin{align*}
1\otimes_A X \cong X \cong X\otimes_A 1.
\end{align*}
Let  $\lambda_G : 1\times G \longrightarrow G$ and  $\rho_G: G \times 1 \longrightarrow G$ be canonical isomorphism.  One has, 
\begin{align*}
(Iso(\lambda_G)^A (s_{\mathcal{E}_F} \times s))(1 \otimes x)& = F\left(\left[ \frac{G\times (1\times G)}{^{(\lambda_G,1)} \Delta(1\times G) ,1 } \right] \right) ((s_{\mathcal{E}_F} \times s)(1\times x))\\
&= F\left(\left[ \frac{G\times (1\times G)}{^{(\lambda_G,1)} \Delta(1\times G) ,1 } \right]  \right)(\mathcal{E}_F \times s(x))\\
&=s(x),
\end{align*}
then 
\begin{align*}
F_+ (Iso(\lambda_G )^A)(\mathcal{E}_{F_+} \times [X,s]_G) =[X,s]_G.
\end{align*}
We can prove 
 \begin{align*}
F_+ (Iso(\rho_G )^A)( [X,s]_G\times \mathcal{E}_{F_+} ) =[X,s]_G
\end{align*}
 in the same way.  This  $\mathcal{E}_{F_+}$  is the identity element.\\
\textbf{  Functoriality }: Let $U$ be an $A$-fibered $(G^\prime, G)$ -biset and let  $V$ be an $A$-fibered  $(H^\prime, H)$-biset such that all stabilizing pairs  are elements of  $\mathcal{S}_+$.  For  $[X,s]\in F_+(G)$ and  $[Y,t]\in F_+(H)$, one has 
\begin{align*}
 F_+(U \otimes_A V)( [X,s]_G\times [Y,t]_H)&=F_+(U\otimes_A V)([X\otimes_A Y, s\times t]_{G\times H} )\\
 &=[(U\otimes_A V)\otimes_{AGH} (X\otimes_A Y), (U\otimes_A V)(s\times t)]_{G^\prime \times H^\prime}.
 \end{align*}
On the other hand 
 \begin{align*}
 F_+(U) ([X,t]_G) \times F_+(V)([Y,t]_H)&=[U\otimes_{AG} X, U(s)] _{G^\prime} \times [V\otimes_{AH} Y, V(t)]_{H^\prime}\\
 &=[(U\otimes_{AG} X) \otimes_A (V\otimes_{AH} Y), U(s)\times V(t)]_{G^\prime \times H^\prime}.
 \end{align*}
Note that, the map  
 \begin{align*}
 (U\otimes_A V)\otimes_{AGH} (X\otimes_A Y)& \longrightarrow (U\otimes_{AG} X) \otimes_A (V\otimes_{AH} Y)\\
 (u\otimes v)\otimes_{AGH} (x\otimes y) &\longmapsto (u\otimes_{AG} x) \otimes (v\otimes_{AH} y)
 \end{align*}
 is an isomorphism of $A$-fibered  $(G^\prime, H^\prime)$-biset.  Now, we will prove that 
\begin{align*}
U(s)\times V(t)((u\otimes_{AH} x)\otimes (v\otimes_{AH} y)) =U(s)(u\otimes_{AH} x) \times V(t)(v\otimes_{AH} y). 
\end{align*}
On the one hand, $U(s)(u\otimes_{AH} x) \times V(t)(v\otimes_{AH} y) $  is equal to  
\begin{align*}
F\left( \left[ \dfrac{G^\prime_ {u\otimes_{AG} x} \times G_x}{ (G^\prime\times G_x)_u, \phi_u} \right] \right) (s(x)) \times F\left(\left[\dfrac{H^\prime_{v\otimes_{AH} y} \times H_y}{(H^\prime \times H_y)_v, \phi_v } \right]  \right) (t(y)). 
\end{align*}
On the other hand
\begin{align*}
U\otimes_A V(s\times t)((u\otimes v)\otimes_{AGH} (x\otimes y)) 
\end{align*}
is equal to 
\begin{align*}
F\left( \left[ \dfrac{(G^\prime\times H^\prime)_{(u\otimes v)\otimes_{AGH} (x\otimes y)} \times (G\times H)_{x\otimes y}}{((G^\prime \times H^\prime) \times (G^\prime \times H^\prime)_{x\otimes y})_{u\otimes v}, \phi_{u\otimes v }}\right]  \right) (s(x)\times t(y) ).
\end{align*}
since 
\begin{align*}
\left[ \dfrac{G^\prime_ {u\otimes_{AG} x} \times G_x}{ (G^\prime\times G_x)_u, \phi_u} \right] \otimes_A \left[\dfrac{H^\prime_{v\otimes_{AH} y} \times H_y}{(H^\prime \times H_y)_v, \phi_v } \right]
\end{align*}
is isomorphic to 
 \begin{align*}
 \left[ \dfrac{(G^\prime\times H^\prime)_{(u\otimes v)\otimes_{AGH} (x\otimes y)} \times (G\times H)_{x\otimes y}}{((G^\prime \times H^\prime) \times (G^\prime \times H^\prime)_{x\otimes y})_{u\otimes v}, \phi_{u\otimes v }}\right],
 \end{align*}
 by   functoriality of $F$, one has  functoriality of  $F_+$. 

\end{proof}
Let $G$ and $H$ be finite groups, and let $D\leq G\times H$, we set 
\begin{align*}
p(D)=p_1(D) \times p_2(D).
\end{align*}
Now we use this notation for the next lemma.
\begin{lem}\label{green^+}
Let $\mathcal{(G,S)}$   satisfy Axioms  $(i)$-$(iii)$,  $(vii)$ and  $(viii)$ in \ref{axioms}, and  the condition $k_2$, set   $\mathcal{D= C}^A(\mathcal{G,S})$,  $\mathcal{D^+=C }^A\mathcal{(G,S^+})$.  If  $F\in \mathcal{F}_{\mathcal{D},R}^{A,\mu}$,  then  $F^+ \in \mathcal{F}_{\mathcal{D}^+, R}^{A,\mu}$.
\end{lem}
\begin{proof}
We define the bilinear product of   $F ^+$.  Let $G$, $H$  be elements of  $\mathcal{G}$,  define bilinear product 
\begin{align*}
\times : F^+(G) \times F^+(H) &\longrightarrow F^+(G\times H)\\
((a_{(L,\lambda)})_{(L,\lambda) \in \mathcal{M}^\prime (G)}, (b_{(K,\beta)})_{(K,\beta) \in \mathcal{M}^\prime (H)}) &\longrightarrow ( c_{(T,\alpha)})_{( T,  \alpha) \in \mathcal{M}^\prime (G\times H)}
\end{align*}
where 
\begin{align*}
  c_{(T,\alpha)} = 
\begin{cases}
 F(Res^{p(T)}_T)( a_{(p_1(T), \lambda)} \times b_{(p_2(T),\beta)}) &\text{ if } \alpha = \lambda \times \beta\\
0 &\text{     otherwise }
\end{cases}
\end{align*}
 and identity element is  $\mathcal{E}_F \in F^+(1)$. Now we will prove that this bilinear product and this identity element  satisfies  the Definition \ref{Green-fiberd}.\\
\textbf{Associativity}:  Let  $G$, $H$, $K$ be  elements of $\mathcal{G}$ and  let  $$\Phi :(G\times H) \times K \longrightarrow  G \times(H\times K)$$ be canonical isomorphism,  for any  $ x=(x_{(L,\lambda)} )_{(L, \lambda) \in \mathcal{M}^\prime (G)} \in F^+(G)$, $y=(y_{(I,\gamma)} )_{(I, \gamma) \in \mathcal{M}^\prime (H)} \in F^+(H)$, and  $z=(z_{(J,\beta)} )_{(J, \beta) \in \mathcal{M}^\prime (K)} \in F^+(K)$, we  set
\begin{align*}
(b_{(T,\tau)} )_{(T, \tau) \in \mathcal{M}^\prime ( H \times K)}= y \times z
\end{align*}
with  
\begin{align*}
b_{(T,\tau)}=  F(Res^{ p(T)}_{T} )(y_{(p_1(T), \gamma) } \times z_{(p_2(T), \beta)}) 
\end{align*}
if   $\tau = \gamma \times \beta$ for some  $\beta \in p_2(T)^\ast $,  $\gamma \in p_1(T)^\ast$ and  $0$ in the otherwise. Now,  we set 
\begin{align*}
a:=(a_{(Q,t) })_{(Q,t)\in \mathcal{M}(G\times (H\times K))}=x\times b
\end{align*}
with
\begin{align*}
a_{(Q,t)} &= F(Res^{ p(Q)}_Q)(x_{(p_1 (Q), \lambda )} \times b_{(p_2(Q), \tau)}) \ \ \ (\text{                                               with } t =\lambda\times \tau )\\
 &= F(Res^{ p(Q)}_Q)(x_{(p_1 (Q), \lambda )} ) \times F(Res^{p_{1,2}(Q) \times p_{2,2}(Q)}_{p_2(Q) })(y_{(p_{1,2}(Q), \gamma)} \times z_{(p_{1,2}(Q)),\beta)} )
\end{align*}
this is equal to 
 \begin{align*}
 &F(Res^{  p(Q)}_Q \circ Res^{p_1(Q) \times ( p_{1,2}(Q) \times p_{2,2}(Q))}_{p_1(Q) \times p_2(Q)} )( x_{(p_1 (Q), \lambda) } \times (y_{(p_{1,2}(Q), \gamma)} \times z_{(p_{1,2}(Q)),\beta)}) )\\
 &= F( Res^{p_1(Q) \times ( p_{1,2}(Q)) \times p_{2,2}(Q))}_{Q} )( x_{(p_1 (Q), \lambda) } \times (y_{(p_{1,2}(Q), \gamma)} \times z_{(p_{1,2}(Q),\beta)}) )
 \end{align*}
with $t=\lambda \times (\gamma \times \beta)$, and where $p_{1,2} (Q)=p_1(p_2(Q)) \leq H$ and $p_{2,2}(Q)=p_2(p_2(Q)) \leq K$. We set  $$c=c(c_{(\overline{Q}, \overline{t})})_{(\overline{Q}, \overline{t})\in \mathcal{M} ((G\times H)\times K)} := F^+(Iso(\Phi)^A)(x\times (y\times z))$$
where  $\overline{Q} = \Phi (Q) $ and  $\overline{t} = t \circ \Phi$  for $(Q,t)\in \mathcal{M}(G\times H) \times K)$. 
One has,   $ c_{\overline{Q}, \overline{t}}$ is equal to 
\begin{align*}
F(Iso (\Phi)^A \circ Res^{p_1(Q) \times ( p_{1,2}(Q) \times p_{2,2}(Q))}_{Q} )( x_{(p_1 (Q), \lambda) } \times (y_{(p_{1,2}(Q), \gamma)} \times z_{(p_{1,2}(Q),\beta)}) ),
\end{align*}
 By  associativity  of  bilinear product of  $F$,  this is equal to $(\overline{Q}, \overline{t})$-coordinate  of  $(x\times y) \times z$.\\
\textbf{Identity element}:  Let $\lambda_G : 1\times G \longrightarrow G$ and $\rho_G: G \times 1 \longrightarrow G$ be canonical isomorphic. For  $y=(y_{(I,\beta)})_{(I,\beta) \in \mathcal{M}^\prime (G)} \in F^+(G)$, we set 
\begin{align*}
(b_{(K,\alpha)})_{(K,\alpha) \in \mathcal{M}^\prime ( G)} = F^+(Iso(\lambda_G) ^A) (\mathcal{E}_{F^+}\times y)
\end{align*}
where
\begin{align*}
b_{(K,\alpha )} =\sum_{\substack{ u \in [A\backslash Iso \lambda_G^A / 1\times  G]\\ K \leq p_1 ((G\times (1\times G))_u) \\ \alpha =1\cdot \beta}} F\left(\left[ \dfrac{K\times K^u}{(K\times (1\times G))_u ,\phi_u}\right]  \right) ( (\mathcal{E}_F \times y)_{(K^u, \beta)}).
\end{align*}
By Definition of  bilinear product $\times$ of $F^+$, one  has  $(\mathcal{E}_F \times y)_{(K^u, \beta)}\neq 0$ if  $K^u = 1\times I$ for some $I\in \sum_{\mathcal{G}} G$, it follows that  $K=I$,  then 
 \begin{align*}
 b_{(K,\alpha )} &= F\left(\left[ \frac{K\times (1\times K)}{^{(\lambda_G, 1)} \Delta(1\times K), 1}\right]   \right) (\mathcal{E}_F \times y_{(K, \alpha)})\\
 &= y_{(K,\alpha)}.
 \end{align*}
 In the same way, we can show that  $F^+(Iso \rho_G ^A)(y\times \mathcal{E}_{F^+})=y$.\\
 \textbf{Functoriality}:  Let $U$ be a $A$-fibered $(G^\prime, G)$ -biset  and let  $V$ be an $A$-fibered $(H^\prime, H)$-biset such that all stabilizing pairs are elements of   $\mathcal{S}^+$.  For  $x=(x_{(K,\alpha)})_{(K,\alpha) \in \mathcal{M}^\prime (G)} \in F^+(G)$ and   $y=(y_{(I,\beta)})_{(I,\beta) \in \mathcal{M}^\prime (H)} \in F^+(H)$, we set 
 \begin{align*}
 b=(b_{(L, \delta)})_{(L,\delta)\in \mathcal{M}^\prime (G^\prime \times H^\prime)}:=F^+ (U\otimes_A V) (x\times y)
 \end{align*}
where  $b_{(L,\delta)} $ is equal to 
\begin{align*}
 \sum_{\substack{u\otimes v \in [A\backslash (U\otimes V) / G\times H]\\ L\leq p_1((G^\prime \times H^\prime) \times (G\times H))_{u\otimes v} \\ \phi_{u\otimes v } \ast s=\delta }} F\left[\left( \frac{L  \times L^{u\otimes v}}{(L\times (G\times H))_{u\otimes v}, \phi_{u\otimes v}}\right)  \right] ((x\times y)_{(L^{u\otimes v},s)}).
\end{align*} 
By Definition of bilinear product of $F^+$, one has  $((x\times y)_{(L^{u\otimes v},s)})\neq 0 $ if only if there exists  $(p_1(L^{u\otimes v}),\lambda)\in \mathcal{M}^\prime(G)$ and $ (p_2(L^{u\otimes v}),\beta)\in \mathcal{M}^\prime (H) $ such that    $s=\lambda\times \beta$. Then  $b_{(L,\delta)} $ is equal to 
\begin{align*}
\sum F\left[\left( \frac{L  \times L^{u\otimes v}}{(L\times (G\times H))_{u\otimes v}, \phi_{u\otimes v}}\right) \circ Res^{p(L^{u\otimes v})}_{L^{u\otimes v}} \right] (x_{p_1(L^{u\otimes v}), \lambda}\times y_{p_2(L^{u\otimes v}),\beta}),
\end{align*} 
where the sum run over  $u\otimes v \in [A\backslash (U\otimes V) / G\times H]$,  $ L\leq p_1\left( ((G^\prime \times H^\prime) \times (G\times H))_{u\otimes v}\right) $ and   $\phi_{u\otimes v } \ast (\lambda \times \beta)=\delta.$
Moreover, by the action of  $U\otimes_A V$, we have $p_1(L^{u\otimes v})= p_1(L)^u$ and  $p_2(L^{u\otimes v})=p_2(L)^v.$\\
On the other hand,  we set  $$a:=(a_{(N, t)})_{(N,t)\in \mathcal{M}^\prime (G^\prime)}= F^+(U) (x)$$   and $$z:=(z_{(W,  s)})_{(W,s)\in \mathcal{M}^\prime (H^\prime)}:= F^+(U) (x),$$ then  $(L,\delta)$-coordinate  of $ F^+(U) (x)\times F^+(V) (y) $ is equal to  
\begin{align*}
F(Res^{p_1(L)\times p_2(L)}_ L) (a_{(p_1(L), t)} \times  z_{(p_2(L), s)})
\end{align*} 
with $\delta = t\times s$, where 
\begin{align*}
a_{(p_1(L), t)}= \sum_{\substack{u\in [A\backslash U / G]\\p_1(L) \leq p_1((G^\prime \times G)_u) \\ \phi_u \ast \kappa=\lambda }} F\left(\left[ \frac{p_1(L)\times p_1(L)^u}{(p_1(L) \times G)_u, \phi_u}\right]  \right) (x_{(p_1(L)^u, \lambda)})
\end{align*}
and 
\begin{align*}
z_{(p_2(L), s)} = \sum_{\substack{ v\in [A\backslash V / H]\\ p_2(L) \leq p_1((H^\prime \times H)_v)\\ \phi_v \ast \beta=\gamma}} F\left(\left[ \frac{p_2(L)\times p_2(L)^v}{(p_2(L) \times H)_v, \phi_v}\right]  \right) (y_{(p_2(L)^v, \gamma)}). 
\end{align*}
Now, note that
\begin{align*}
 \left[ \frac{p_1(L)\times p_1(L)^u}{(p_1(L) \times G)_u, \phi_u}\right]  \otimes  \left[ \frac{p_2(L)\times p_2(L)^v}{(p_2(L) \times H)_v, \phi_v}\right]
\end{align*}
is equal to 
\begin{align*}
 \left[ \frac{L  \times L^{u\otimes v}}{(L\times (G\times H))_{u\otimes v}, \phi_{u\otimes v}}\right] \circ Res^{p(L^{u\otimes v})}_{L^{u\otimes v}} 
\end{align*}
 By the proof of  Theorem  \ref{F^+ conjuga} and the functionality of  bilinear product $\times$  of  $F$, ones has the $(L,\delta)$-coordinate of $F^+(U) (x)\times F^+(V) (y)$ is equal to the $(L,\lambda)$-coordinate of $F^+(U\otimes _A V)(x\times y)$.
\end{proof}

\begin{thm}
Let $\mathcal{(G,S)}$ satisfy Axioms $(i)–(iii)$ in \ref{axioms} and set $\mathcal{ D := C}^A \mathcal{(G, S)}$. Further,
let $F\in \mathcal{F}_{\mathcal{D},R}^{A,\mu}$.

\begin{enumerate}[(a)]
\item Assume that $\mathcal{(G, S)}$ also satisfies Axioms $(iv)$ and $(vi)$ in \ref{axioms} and set  $\mathcal{D_+=C}^A \mathcal{(G,S_+)}$. Then, $F_+ \in \mathcal{F}_{\mathcal{D}_+, R}^{A,\mu}$ and one obtains a functor $-_+:  \mathcal{F}_{D R}^{A,\mu} \longrightarrow \mathcal{F}_{\mathcal{D}_+, R}^{A,\mu}$. 

\item
 Assume that $\mathcal{(G, S)}$  also satisfies Axiom $(vii)$ and $(viii)$ in \ref{axioms} and assume that for any $G$ and $H$ in $\mathcal{G}$,  the condition $k_2$ and set $\mathcal{D^+=C}^A \mathcal{(G,S^+})$. Then $F^+ \in \mathcal{F}_{\mathcal{D}^+, R}^{A,\mu}$
and one obtains a functor $$-^+:  \mathcal{F}_{D R}^{A,\mu} \longrightarrow \mathcal{F}_{\mathcal{D}^+, R}^{A,\mu}.$$
\item  If $\mathcal{(G, S)}$ also satisfies Axioms $(iv)$, $(vi)$, $(vii$) and $(viii)$ in \ref{axioms}  and the condition $k_2$, then the mark morphism
$$m_F : F_+ \longrightarrow F^+$$ is a morphism in  the category $F^{ A, \mu}_
{\mathcal{E},R}$, where $\mathcal{E = D_+ = D^+}$.
\end{enumerate}
\end{thm}

\begin{proof}
We have $(a)$ and $(b)$ are consequence of  Lemma \ref{green_+} and Lemma \ref{green^+}.
Now, we will show that $(c)$. Let $G, H\in \mathcal{G}$, we need to prove that the next diagram commute 
\begin{equation*}
\xymatrix{
F_+(G)\times F_+(H)\ar[r]_{   m_{F,G}\times m_{F,H} } \ar[d]_{\times } &  F^+(G)\times F^+(H) \ar[d]^{\times}\\
F_+(G\times H)\ar[r]_{m_{F,G\times H}} & F^+(G\times H) 
}
\end{equation*}
Let $[X,s]\in F_+(H)$ and $[Y,t]\in F_+(H)$.  First,  the $(T,\alpha)$-component  \-of $m_{F, G\times H}([X\otimes_A Y, s\times t])$ is equal to 
\begin{align*}
\sum_{\substack{x\otimes y\in [X\otimes Y]\\ (T\alpha)\leq (G_x\times H_y, \phi_x \times \phi_y)}} F(Res^{G_x \times H_y}_T ) (s(x)\times t(y))
\end{align*}
On the other hand, the $(T,\alpha)$-component of $m_{F,G}([X.,s]) \times m_{F,H}([Y,t])$ is equal to 
\begin{align*}
F(Res^{p_1(T)\times p_2(T)}_T)(m_{F,G}([X,s])_{(p_1(T),\lambda)}\times m_{F,H}([Y,t])_{(p_2(T), \beta) })
\end{align*}
where $\alpha = \lambda \times \beta$, $m_{F,G}([X,s])_{(p_1(T),\lambda)}$ is the $(p_1(T),\lambda)$-component of $m_{F,G}([X,s])$ and $m_{F,H}([Y,t])_{(p_2(T), \beta )}$ is the  $(p_2(T), \beta )$-component of 
$m_{F,H}([Y,t])$, by definition of mark morphism, this is equal to 
\begin{align*}
\sum_{\substack{x\in [X]\\y \in [Y]\\(p_1(T),\lambda)\leq (G_x, \phi_x)\\(p_2(T), \beta )\leq (H_y,\phi_y)}} F(Res^{G_x\times H_y}_T)(s(x)\times t(y)),
\end{align*}
one has, the set $\lbrace x\otimes y \in X\otimes_A Y \mid x\in [X] \text{ and }  y \in [Y]\rbrace$ is a set of representatives of $A\backslash  X\otimes_A Y /G\times H$, then $(T,\alpha)$-component of $m_{F,G}([X.,s]) \times m_{F,H}([Y,t])$  is equal to the $(T,\alpha)$-component of $m_{F, G\times H}([X\otimes_A Y, s\times t])$.
\end{proof}

\section{example } \label{example fibre}
In this section, we will define some fibered biset functors  and we will study the relationship between  their  construction $-_+$ previously defined  and  these functors. The first functor we will define is the character ring functor (see the  section  11B. of  \cite{fibered}) and the second functor is the global representation fibered ring functor, this functor is  inspired  by  the global representation ring functor. 

\subsection{The Functor $\mathbb{K}R_\mathbb{C}$}
Let $A=\mathbb{C}^\times$. For a finite group $G$, we denote the ring of characters of $\mathbb{C}[G]$-modules by $R_{\mathbb{C}}(G)$. The linearization mapping can be defined as follows:
\begin{align*}
lin_G : B^{\mathbb{C}^\times} (G) &\longrightarrow R_{\mathbb{C}}(G)\\
[H,\phi]G &\longmapsto Ind^G_H (\phi).
\end{align*}
By the Brauer Induction Theorem, this morphism is surjective. From Theorem 1.1 of \cite{serge-biset} and Corollary 2.5 of \cite{fibered},

\begin{align*}
lin_{G\times H} ([X])\otimes_{\mathbb{C} H}  lin_{H\times K} ([Y]) =lin_{G\times K} ([X\otimes_{\mathbb{C} H} Y])
\end{align*}
in $R_{ \mathbb{C}} (G\times K)$. Consequently, the relation $G \longrightarrow R_\mathbb{C} (G)$ defines a functor of $\mathbb{C}^\times$-fibered sets, denoted by $R_{\mathbb{C}}$. This functor sends a basic element $\left[ \frac{G \times H}{ U,\phi} \right]$ of $B^{\mathbb{C}^\times} (G,H)$ to the mapping:
 \begin{align*}
R_\mathbb{C} \left( \left[  \frac{G \times H}{ U,\phi} \right]\right)  :R_\mathbb{C} (H)& \longrightarrow R_\mathbb{C} (G)\\
 [M] &\longmapsto [Ind_U^{G\times H} (\mathbb{C}_\phi) \otimes_{\mathbb{C}H} M],
 \end{align*}
where $\mathbb{C}_\phi$ is the set $\mathbb{C}$ where $U$ acts by the right through $\phi$.

\subsection{The global representation fibered ring functor } 
Now, we define the global representation fibered ring functor.  It is important, recall the global
representation ring is defined  in  \cite{raggiglobal} and  this ring   have  a biset functor structure (see  \cite{karleyTesis}).   The next functor   is inspired in the  the global representation ring  functor.  If an analogy is made, it would be like the Burnside functor and the  Burnside fibered  functor. 
\begin{defi}
    
Given $G$ a group and $X$ a $G$-set, a $\mathbb{C}G$-module $V$ is said to be $X$-graded if
\begin{align*}
    V=\bigoplus_{x\in X} V_x,
\end{align*}
where each $V_x$ is a $\mathbb{C}$-vector subspace such that $gV_x=V_{gx}$ for every $g\in G$ and $x\in X$.
\end{defi}
\begin{defi}
Let be $G$ a finite group. We denote by  $\aleph_G$ the category whose objects are
the pairs  $( X, V )$  where  $X$
is an  $A$-fibered  $G$-set and    $V$  is an $X/A$-graded $\mathbb{C}G$-module,  where  $X/A$ is the set of all   $A$-orbits of   $X$. The morphisms in $\aleph_G$  from $(X, V )$ to $(Y, W)$
are pairs of morphisms $(\alpha, f)$ where $\alpha : X \longrightarrow Y $ is a morphism of $A$-fibered $G$-sets, $f : V\longrightarrow  W$ is a morphism of $\mathbb{C}G$-modules and $f(V_x)  \subseteq W_{\alpha(x)}$ for all $x$ in $X$. The composition is the natural composition  and the identity  morphism is the pair $(Id_X, Id_V)$, where $1_X$ is the identity function of the $G$-set $A$-fibered $X$, and $1_V$ is the identity function of the $\mathbb{C}G$-module $V$. We denote  $[ X, V ]$  the isomorphism class of   $( X, V )$ in $\aleph_G$ .
\end{defi}
The category  $\aleph_G$   has a coproduct defined  by  $$(X,V)\oplus (Y, W)=(X\sqcup Y, V\oplus W),$$ where $\sqcup$ is the disjoint union  of the sets and  $\oplus$ is the direct sum of modules. 
We defined  $T^A(G):=G_0(\aleph_G, \sqcup)$  the Grothendieck group, $T^A(G)$ has structure of ring with the product $[X, V]\cdot[Y,W]:=[X\times Y, V\otimes W]$
\begin{defi}
We define the global $A$-fibered  representation ring of the group $G$ as the quotient
\begin{align*}
\Game ^A (G) = \dfrac{T^A(G)}{ \left\langle [X,V\oplus W]-[X,V]-[X,W]\right\rangle_\mathbb{Z} }.
\end{align*}
We also write $[H, V ]$ to denote elements of $\Game^A(G)$.
\end{defi}
\begin{defi}
We define  the global $A$-fibered  representation ring functor  
\begin{align*}
\Game^A :R\mathcal{C}^A &\longrightarrow R\text{-Mod}\\
G&\longmapsto \Game^A(G)
\end{align*}
Let   be $G$,   $H$ objects of $ R\mathcal{C}^A$, a let be   $U$ an  $A$-fibered $(G,H)$-biset  element of  $Hom_{\mathcal{D}} (H, G)$,   we define the map 
\begin{align*}
\Game^A(U): \Game^A(H) &\longrightarrow \Game ^A(G)\\
[X,V]_H&\longmapsto [U\otimes_{AH} X, \mathbb{C}U \otimes_{ \mathbb{C} AH } V]_G
\end{align*}
where  $A$ act over  $V$  by the trivial action   $V$. 
The group $\Game(G)$ can be endowed with a multiplicative structure as follows:
\begin{align*}
[X, V]\cdot [Y,W]:=[X\otimes_A Y, V\otimes_\mathbb{C} W].
\end{align*}

Now, let's proceed with the definition.
\end{defi} 
The functor $\Game^A$ is well-defined  such that   the  tensor product of $A$-fibered biset respects the isomorphism classes of $A$-fibered bisets and the direct sum of  modules also with the isomorphism classes of  modules. If $V=\oplus_{x\in X/A} V_x$, one has  
\begin{align*}
\mathbb{C}U \otimes_{ \mathbb{C} AH } V  \cong \bigoplus_{u\otimes x\in U\otimes_{AH} X} u\otimes V_x. 
\end{align*}
where  $ u\otimes V_x:=\lbrace u\otimes v_x \mid v_x \in V_x\rbrace$ is a $\mathbb{C}$-vector subspace. 

\subsection{The functor $ R^{ \mathbb{C}^\times}_{\mathbb{C},+} $}
Let $(\mathcal{G, S})$ as before such that  satisfying Axioms $(i)$-$(iii)$.  To compute the $- _+$ construction of $R_\mathbb{C}$, first  we find the $(\mathbb{C}^\times, G)$-invariant sections of  $R_\mathbb{C}$  over  $\mathbb{C}^\times$ \\ 
Let be   $X$  a $A$-fibered $G$-set  such that  $G_x$ is an element of  $\mathcal{G}$ for all  $x\in X$, then one section of  $R_\mathbb{C}$   over  $X$ is a function  
\begin{align*}
s: X\longrightarrow \bigoplus_{x\in [ X/A]} R_\mathbb{C}(G_x)
\end{align*}
such that $s(x)\in R_\mathbb{C}(G_x)$  for all  $x\in X$. \\
Then the elements of  $R^{ \mathbb{C}^\times}_{\mathbb{C}, +} (G)$ are  the pairs $[X,s]$, where  $X$ is a  $\mathbb{C}^\times$-fibered  $G$-set   and   $s$  is an  $\mathbb{C}^\times \times G$-invariant  section, this means   $^gs(ax)=s(x)$   for all $(a,g)\in \mathbb{C}^\times \times G$.  We define  the  $\mathbb{C}G$-module
\begin{align*}
V_s =  Ind^G_{G_x} s(x)= \bigoplus_{g\in [G/G_x]} g\otimes s(x).
\end{align*}
We can associate the pair $[X, s]_G$ with the pair $[X, V_s]$ in $\Game^A(G)$.

\begin{prop}
The $-_+$  construction of the  $\mathbb{C}^\times $-fibered biset functor  $R^{\mathbb{C}^\times}_\mathbb{C}$, denote  by  $R^{\mathbb{C}^\times}_{\mathbb{C},+}$,  is the  global representation $\mathbb{C}^\times$-fibered ring functor  $\Game^{\mathbb{C}^\times}$. 
\end{prop}
\begin{proof}
  Let be $U $   an  $\mathbb{C}^\times$-fibered $(G,H)$-biset such that  all stabilizer pairs of $U$ are elements  of  $\mathcal{S}(G,H)$,   we define the map 
\begin{align*}
R_{\mathbb{C},+}\left( [U]  \right) : R_{\mathbb{C},+}(H) &\longrightarrow R_{\mathbb{C},+}(G)\\
[X, s]_H &\longrightarrow[ U \otimes_{\mathbb{C}^\times H}X, U(s)]_G
\end{align*}
where
\begin{align*}
U(s)(u\otimes x) &= R_\mathbb{C} \left( \dfrac{ G_{u\otimes x} \times H_x}{(G\times H_x)_u, \phi_u}\right) (s(x))\\
&=[Ind^{G_{u\otimes x} \times H_x}_{(G \times H_x)_u} (\mathbb{C}_{\phi_u } )\otimes _{\mathbb{C}^\times H_x }s(x)]
\end{align*}
 note that 
\begin{align*}
Ind^{G_{u\otimes x} \times H_x}_{(G \times H_x)_u} (\mathbb{C}_{\phi_u})  = \mathbb{C} \left[ \dfrac{G_{u \otimes x} \times H_x }{ ( G\times H_x)_u, \phi_u }\right] ,
\end{align*} 
then
 \begin{align*}
V_{ U(s)} = &\bigoplus_{u\otimes x \in [ U \otimes_{\mathbb{C}^\times  H }  X/ {\mathbb{C}^\times}]} \mathbb{C} \left[ \dfrac{G_{u \otimes x} \times H_x }{ ( G\times H_x)_u, \phi_u }\right]  \otimes_{AH_x} s(x)\\
=&\bigoplus_{u\otimes x \in [ U \otimes_{\mathbb{C}^\times  H }  X/ {\mathbb{C}^\times}]} u\otimes_{AH_x} s(x)\\
=&\bigoplus_{u\otimes x \in [( U/\mathbb{C}^\times) \otimes_{ A H}  (X/ {\mathbb{C}^\times)}]} u\otimes_{AH_x} s(x).
\end{align*}  
On the other hand, one has  
\begin{align*}
\mathbb{C}U \otimes_{\mathbb{C}^\times H} V_s = \mathbb{C}U \otimes_{\mathbb{C}^\times H}  \bigoplus_{x\in [X/A]} s(x)
\end{align*}
by the Definition 3.1 of  \cite{karleyTesis},  one has  
\begin{align*}
\mathbb{C}U \otimes_{\mathbb{C}^\times H} V_s  = &\bigoplus_{u\otimes x \in [U\otimes_{\mathbb{C}^ \times H } (X/\mathbb{C}^\times)] } u\otimes s(x)\\
= &\bigoplus_{u\otimes x \in [(U/\mathbb{C}^\times )\otimes_{A H } (X/\mathbb{C}^\times)] } u\otimes s(x)\\
\end{align*}
We can define the map
\begin{align*}
R_{\mathbb{C}, +}^{\mathbb{C}^ \times} (G) &\longrightarrow \Game^{\mathbb{C}^\times} (G)\\
[X,s]_G &\longmapsto  [X,V_s]_G.
\end{align*}
By the last remark,  we can extend  this map  to natural transformation  and by the definition of  the functors  $R_{\mathbb{C}, +}^{\mathbb{C}^ \times} $ and  $ \Game^{\mathbb{C}^\times}$, this natural transformation is injective and  suryectiva   for all evaluation.  Thus,    the functor  $R_{\mathbb{C}, +}^{\mathbb{C}^ \times}$ is isomorphic  to  $\Game^{\mathbb{C}^\times}$.

 \end{proof}

\section{Adjunciones} \label{adjuncion}
In this section, we will show that the $-_+$ construction has an adjoint.
Let  $R$ be  a commutative ring with unit,  and let  $\mathcal{(G,S)}$  satisfy axioms $(i)$, $(ii)$, $(iii)$, $(iv)$ in \ref{axioms}. We will denote $\mathcal{D=C}^A\mathcal{(G,S)}$.
\begin{defi} 
For all  $G$, $H$ in $\mathcal{G}$,  we set 
$$\mathcal{S}_-(G,H):=\lbrace (D,\phi)\in \mathcal{S}(G,H) \mid p_1(D)=G\rbrace.$$ 
\end{defi}

\begin{lem}
With the previous hypotheses, we have that the pair $\mathcal{(G,S_-)}$ satisfies axioms $(i)$ to $(iv)$ of \ref{axioms}. Moreover, $\mathcal{(D_-)+=D+}$.
\end{lem}
 \begin{proof}
by definition of  $\mathcal{S}_-$, and the Axioms of   $(\mathcal{G,S})$   and the Mackey formula,   one has $(\mathcal{G,S_-)}$  satisfies  axioms  $(i)$, $(ii)$ and $(iii)$  of \ref{axioms}. 
For the axiom  $(iv)$. Let $G$  and $H$ elements of  $\mathcal{G}$,  for any   $(D,\phi)\in \mathcal{S}_-(G,H)$, one has  $(D,\phi) \in \mathcal{S}(G,H)$  and $p_1(D)=G$.  By the axioms  $(iv)$ of  $(\mathcal{G,S})$, one has for any  $K\in \sum_\mathcal{G} (H)$,   we have $D\ast K \in \mathcal{G}$,   $(D\ast \Delta(K), \phi \ast 1) \in \mathcal{S}(D\ast K, K) $ and note that    $p_1(D\ast \Delta(K)) = D\ast K$, e
the  $(D\ast \Delta(K), \phi \ast 1) \in \mathcal{S_-}(D\ast K, K) $. Thus,  the data  $(\mathcal{G,S_-})$ satisfies the axioms  $(i)$, $(ii)$, $(iii)$ and  $(iv)$. We define the new category  $\mathcal{D_-=C}^A(\mathcal{G,S_-})$. \\ 
 By definition of $\mathcal{S}_-$, one has  $(\mathcal{D_-})_+ \subseteq \mathcal{D}_+$.  On the other hand,  let  $(D,\phi) \in \mathcal{S}_+(G,H)$, then  $p_1(D) \in \mathcal{G}$  and  $ (D,\phi) \in \mathcal{S} (p_1(D),H)$ .i.e.  $(D,\phi)  \in \mathcal{S_-} (p_1(D),H)$,   therefore  $(D,\phi)  \in (\mathcal{S}_-)_+ (G,H)$. Thus,  $(\mathcal{D_-})_+=\mathcal{D}_+$
 \end{proof}
We can note that  $\mathcal{D_-:=C}^A( \mathcal{S_-, G}) \subseteq \mathcal{D\subseteq D_+}$.\\
Let  $F\in \mathcal{F}^A_{\mathcal{D_-}, R}$ and let  $G\in \mathcal{G}$,  we define  the map  
\begin{align*}
\eta_{F,G}:F(G)&\longrightarrow Res^{\mathcal{D}_+}_{\mathcal{D}_-}(F_+(G))\\
a&\longmapsto [G,1,a]_G
\end{align*}
\begin{lem}
Suppose that  $\mathcal{(G, S}) $ satisfies the above hypotheses and the axioms $(vi)$.  Then, the functor  $-_+ : \mathcal{F}^A_{\mathcal{D_-}, R}\longrightarrow \mathcal{F}^A_{\mathcal{D_+}, R} $ is the left adjunction of the restriction functor, $ Res^{\mathcal{D}_+}_{\mathcal{D}_-}: \mathcal{F}^A_{\mathcal{D}_-, R} \longrightarrow \mathcal{F}^A_{\mathcal{D}_+,R }$.
\end{lem}
\begin{proof}
Let  $F\in \mathcal{F}^A_{\mathcal{D_-}, R}$ and $M\in  \mathcal{F}^A_{\mathcal{D_+}, R}$. We will prove that the  map 
\begin{align*}
Hom_{\mathcal{D}_+}(F_+, M)& \longrightarrow Hom_{\mathcal{D}_-}(F, Res^{\mathcal{D_+} }_\mathcal{D_-} (M))\\
\phi&\longmapsto \phi \circ \eta_F
\end{align*}
 is an  isomorphism. First, we will prove that the map is well-defined, that is, that $\phi\circ \eta_F$  is a natural transformation of $F$ 
 to $Res^{\mathcal{D_+}}_{\mathcal{D_-}} M$.  Let  $G$, $H$ finite group in  $\mathcal{G}$ and $(D,\phi)\in S_-(G, H)$,  we will prove that the next diagram is commutative 
  \begin{equation*}
 \xymatrix{ F(H) \ar[r]^{\varphi_H \circ\eta_{F,H}} \ar[d]_{F(X)} & M(H) \ar[d]^{Res^{\mathcal{D_+}}_{\mathcal{D_-}} M(X)} \\ F(G) \ar[r]_{\varphi_G \circ \eta_{F,G}} &  M(G).}
 \end{equation*}
For  $a\in F(H)$ and  $X=\dfrac{G\times H}{D, \phi}$, 
 note that $D \ast H = G$ and $\phi \ast 1(g) = \phi(g,h) 1(h) = \phi(g,h)$ for all $h \in H$ such that $(g,h) \in D$. Therefore, $\phi$ depends only on the first coordinate of the elements in $D$. By the previous diagram, we have that,
\begin{align*}
M(X)\varphi_H(\eta_{F,H} (a))&= M(X) \varphi_H ([H, 1 , a]_H)\\
&=M(tw_{\phi_1}) M \left( \left[  \dfrac{G\times H}{D,1}\right] \right) (\varphi_H([H,1,a]_H))\\
&=M(tw_{\phi_1})\left[ \varphi_G(F_+ \left( \left[  \dfrac{G\times H}{D,1}\right] \right) ([H,1,a]_H)) \right] \\
&=M(tw_{\phi_1})([G,1,F \left( \left[ \dfrac{G\times H}{D,1}\right] \right) (a)]_G).
\end{align*}
On the other hand,
 \begin{align*}
\varphi_G\circ \eta_{F,G} (F(X)(a)) &= [G,\phi, F(X)(a)]_G\\
&=\varphi_G \left(  F_+(tw_{\phi_1}) ([G,1, F \left( \left[ \dfrac{G\times H}{D,1}\right] \right) (a)]_G \right) \\
&=M(tw_{\phi_1})(\varphi_G([G,1, F \left( \left[ \dfrac{G\times H}{D,1}\right] \right) (a)]_G).
 \end{align*}
 Thus, $\varphi \circ  \eta_F $  is a natural transformation. \\
 Now we will prove that the function is injective. Let $\varphi$  and  $ \varphi^\prime \in Hom_{\mathcal{D}_+}(F_+, M)$such that  $\varphi \circ \eta_F=\varphi^\prime \circ \eta_F$. Given a finite group $G\in \mathcal{G}$,  and a basic element   $[H, \psi,a]_G$  of  $F_+(G)$, we have  that  
 \begin{align*}
 \varphi_G( [H, \psi,a]_G) &= \varphi (Ind^G_{H,\psi} [H,1,a]_H)\\
 &=\varphi_G (Ind^G_{H,\psi} \circ \eta_{F, H} (a))\\
 &=M( Ind^G_{H,\psi} )\left( \varphi_H (\eta_{F, H} (a))\right) \\
 &=M( Ind^G_{H,\psi} )\left( \varphi_H ^\prime(\eta_{F, H} (a))\right) \\
 &=\varphi_G^\prime( [H, \psi,a]_G) .
  \end{align*}
Therefore,  $\varphi=\varphi^\prime$,  which means that the previous morphism is injective. Now we will prove that the morphism is surjective. Given  $\psi \in  Hom_{\mathcal{D}_-}(F, Res^{\mathcal{D_+} }_\mathcal{D_-} (M))$. For every,  $G\in \mathcal{G}$   let's consider the function.
 \begin{align*}
 \varphi_G : F_+(G)  &\longrightarrow M(G)\\
[X,s]& \longmapsto\sum_{x\in [G\backslash X /A]} M(Ind^G_{G_x, \phi_x}) (\psi_{G_x} (s(x))).
 \end{align*}
The first thing, we will note  that this definition does not depend on the set of representatives of $[G\backslash X/A]$. This is because if we consider  another representative of the $(G,A)$-orbit of $x$, its stabilizer is a $G$-conjugate of $G_x$, and $s$ is a $(G,A)$-invariant section.  Furthermore, let us note that given $[X, s]$, $[Y, t]$, $[W, r + s]$ in $\Gamma_F(G)$, we have that 
\begin{align*}
\varphi_G([X\sqcup, Y, s+t ])=  \sum_{z\in [G\backslash X \sqcup Y /A]} M(Ind^G_{G_z, \phi_z}) (\psi_{G_z} (s + t(z)))\\
\end{align*}
this is equal to 
\begin{align*}
\sum_{z\in [G\backslash X  /A]} M(Ind^G_{G_x, \phi_z}) (\psi_{G_z} (s(z))) + \sum_{z\in [G\backslash Y /A]} M(Ind^G_{G_z, \phi_z}) (\psi_{G_z} (t(z)))\\\\
\end{align*}
 Therefore,  $\varphi_G([X\sqcup, Y, s+t ]=\varphi_G([X, s ]) + \varphi_G([ Y, t ])$, and since  $M(Ind_{H,\phi}^G)$ is an  $R$-liner function, we have that  
 \begin{align*}
 \varphi_G( [W,r+s]) &= \sum_{w\in [G\backslash W  /A]} M(Ind^G_{G_w, \phi_w}) (\psi_{G_w} ( r+s(w))\\
 & = \sum_{w\in [G\backslash W  /A]} M(Ind^G_{G_w, \phi_w}) (\psi_{G_w} ( s(w))+ M(Ind^G_{G_w, \phi_w}) (\psi_{G_w} ( s(w))\\
 &= \varphi_G([W,r])+\varphi_G([W, s]),
 \end{align*}
 in consequence, the function  $\varphi_G$  is well-defined in  $F_+(G)$. \\
 For any  $G\in \mathcal{G}$   and for any    $a\in F(G)$. we have, 
 \begin{align*}
 \varphi_G(\eta_{F,G}(a))&=\varphi_G([G,1,a]_G)\\
 &=\varphi_G( Ind^G_{G,1})( \psi (a))\\
 &=\psi(a).
 \end{align*}
Finally, we will prove  that the morphism  $\varphi:=(\varphi_G)_{G\in \mathcal{G}}$ is a natural transformation in  $Hom_{\mathcal{D}_+}(F_+, M)$.\\
Let  $G$, $H$  be  finite groups and  let be   $X:=\left[ \dfrac{G\times H}{D,\phi}\right] $ an element of  $Hom_{\mathcal{D_-}}(H,G)$. Now, we will prove that the following diagram commutes: 
 \begin{equation*}
 \xymatrix{ F_+(H) \ar[r]^{\varphi_H} \ar[d]_{F_+(X)} & M(H) \ar[d]^{M(X)} \\ F_+(G) \ar[r]_{\varphi_G} & M(G) }
 \end{equation*}
Let  $[K,f,a]_H \in F_+(H)$, on one hand, we have that 
 \begin{align*}
\varphi_H( F_+(X)([K,f,a]_H))= \sum_{h\in [ p_2(D) \backslash H/K]}  \varphi_G([D\ast {^h K}, \phi \ast {^hf}, b_h]_G),
 \end{align*}
 where   (referring to equation \ref{F_+[trans](---)})
 \begin{align*}
 b_h=F\left( \left[  \dfrac{D \ast {^hK} \times {^hK}}{ D\ast \Delta(^hK ), \phi } \right]\right) (^ha).
 \end{align*}
Thus  
 \begin{align*}
 \varphi_H( F_+(X)([K,f,a]_H))= \sum_{h\in [ p_2(D) \backslash H/K]} M(Ind^G_{D\ast {^hK}, \phi \ast{^hf}}) (\psi_{D\ast {^hK}} (b_h))
 \end{align*}
 since $\psi$ is a natural transformation, we have that 
 \begin{align*}
 \varphi_H( F_+(X)([K,f,a]_H))&= \sum_{h\in [ p_2(D) \backslash H/K]} M \left(Ind^G_{D\ast {^hK}, \phi \ast{^hf}} \otimes \left[\dfrac{D \ast {^hK} \times {^hK}}{ D\ast \Delta(^hK ), \phi }  \right]\right)  (\psi_{ {^hK}} (^ha))\\
 &= \sum_{h\in [ p_2(D) \backslash H/K]} M(\left( \left[\dfrac{G\times ^hK}{D\ast \Delta(^hK), \phi \ast {^hf}}\right] \right)  (\psi_{^hK} ({^ha})).
 \end{align*}
 On the other hand, we have that  
  \begin{align*}
  M(X)(\varphi_H([K,f,a]_H))&=M(X)\left( M(Ind^G_{K,f})(\psi_{K} (a)) \right) \\
  &= \sum_{h\in [ p_2(D) \backslash H/K]} M(\left( \left[\dfrac{G\times ^hK}{D\ast \Delta(^hK), \phi \ast {^hf}}\right] \right)  (^h\psi_{K} (a)).
  \end{align*}
 Since  $\psi$ is a natural transformation, we have that    $^h\psi(a)=\psi(^ha)$. Therefore,  $\varphi$  is a natural transformation, from which we conclude that 
 $$Hom_{\mathcal{D}_+}(F_+, M) \cong Hom_{\mathcal{D}_-}(F, Res^{\mathcal{D_+} }_\mathcal{D_-} (M).$$
\end{proof}

A natural question that arises is: Does the functor $-^+$ have  right adjoint functor? In the case of bisets, it was shown by Boltje, Raggi, and Valero in \cite{constructios+} that the functor $-^+$ does have an adjoint functor.

Please refer to Journal-level guidance for any specific requirements.

\bibliographystyle{plain}
\bibliography{biblio1}

\begin{thebibliography}{10}

\bibitem{canonicalinductionformulae}
Robert Boltje.
\newblock A general theory of canonical induction formulae.
\newblock {\em Journal of Algebra}, 206(1):293--343, 1998.

\bibitem{fibered}
Robert Boltje and Olcay Co{\c{s}}kun.
\newblock Fibered biset functors.
\newblock {\em Advances in Mathematics}, 339:540--598, 2018.

\bibitem{constructios+}
Robert Boltje, Gerardo Raggi-C{\'a}rdenas, and Luis Valero-Elizondo.
\newblock The $-_+$ and $-^+$ constructions for biset functors.
\newblock {\em Journal of Algebra}, 523:241--273, 2019.

\bibitem{unitp-grupos}
Serge Bouc.
\newblock The functor of units of burnside rings for p-groups.
\newblock {\em Commentarii Mathematici Helvetici}, 82(3):583--615, 2007.

\bibitem{serge-biset}
Serge Bouc.
\newblock {\em Biset functors for finite groups}.
\newblock Springer, 2010.

\bibitem{bisetscategory}
Serge Bouc.
\newblock Bisets as categories and tensor product of induced bimodules.
\newblock {\em Applied Categorical Structures}, 18(5):517--521, 2010.

\bibitem{karleyTesis}
Karley~Tatiana Cardona~Echenique et~al.
\newblock El funtor global de representaciones como funtor de biconjuntos de
  {G}reen.
\newblock 2016.

\bibitem{dressRing}
Andreas Dress.
\newblock The ring of monomial representations {I}. {Structure} theory.
\newblock {\em Journal of Algebra}, 18:137--157, 1971.

\bibitem{raggiglobal}
Alberto~G Raggi-C{\'a}rdenas and Luis Valero-Elizondo.
\newblock Global representation rings.
\newblock {\em Journal of Algebra}, 441:426--440, 2015.

\bibitem{MobiusGCRota}
Gian-Carlo Rota.
\newblock On the foundations of combinatorial theory {I}. {T}heory of
  {M}{\"o}bius functions.
\newblock {\em Zeitschrift f{\"u}r Wahrscheinlichkeitstheorie und verwandte
  Gebiete}, 2(4):340--368, 1964.

\end{thebibliography}
\end{document}